\documentclass[11pt]{amsart}
\newcommand{\W}{\Phi}

\usepackage{tikz}
\usepackage{pgfplots}
\usepackage{amsthm}
\usepackage{amsxtra}
\usepackage{amssymb}
\usepackage{graphicx}
\usepackage{comment}
\usepackage{url}
\usepackage{color}
\usepackage{amsfonts}
\usepackage{textcomp}

\usepackage{caption}
\usepackage{subcaption}
\DeclareMathAlphabet{\mathpzc}{OT1}{pzc}{m}{it}
\usepackage{lmodern}

\usepackage{sidecap}

\DeclareMathOperator*{\WFh}{WF_h}
\DeclareMathOperator*{\loc}{loc}
\DeclareMathOperator*{\comp}{comp}

\DeclareMathOperator*{\Arg}{Arg}

\usepackage{chngcntr}

\graphicspath{{Figures/}}

\newtheorem{theorem}{Theorem}

\numberwithin{prop}{section}

\numberwithin{corol}{section}

\newtheorem{lemma}{Lemma}
\numberwithin{lemma}{section}

\numberwithin{conjecture}{section}

{\theoremstyle{definition}

\numberwithin{defin}{section}
}
\numberwithin{figure}{section}

\renewcommand{\Re}{\mathop{\rm Re}\nolimits}
\renewcommand{\Im}{\mathop{\rm Im}\nolimits}

\newcommand{\pO}{{\partial\Omega}}

\newcommand{\bl}{\begin{flushleft}}
\newcommand{\el}{\end{flushleft}}
\newcommand{\br}{\begin{flushright}}
\newcommand{\ert}{\end{flushright}}
\newcommand{\bc}{\begin{center}}
\newcommand{\ec}{\end{center}}

\newcommand{\recip}[1]{\frac{1}{#1}}

\newcommand{\complex}{\mathbb{C}}

\newcommand{\numList}{\begin{enumerate}}
\newcommand{\enumList}{\end{enumerate}}

\newcommand{\e}{\epsilon}

\newcommand{\re}{\mathbb{R}}

\newcommand{\la}{\langle}
\newcommand{\ra}{\rangle}

\newcommand{\resd}{\Lambda(h)}

\newcommand{\asec}{\operatorname{arcsec}}

\newcommand{\mc}[1]{\mathcal{#1}}

\theoremstyle{remark}
\newtheorem{remark}{Remark}

\newcommand{\Deltad}[1]{\Delta_{V,#1}}

\renewcommand{\O}[1]{\mathpzc{O}_{#1}}
\renewcommand{\o}[1]{\mathpzc{o}_{#1}}

\usepackage{xcolor}
\selectcolormodel{gray}
\title [Resonances for Thin Barriers on the Circle]{Resonances for Thin Barriers on the Circle}
\author[J. Galkowski]{Jeffrey Galkowski}
\address{Mathematics Department, Stanford University, Stanford, 
CA, USA, 94305}
\email{jeffrey.galkowski@stanford.edu}

\begin{document}
\begin{abstract}
We study high energy resonances for the operator $-\Deltad{\pO}:=-\Delta+\delta_{\partial\Omega}\otimes V $ when $V$ has strong frequency dependence. The operator $-\Deltad{\pO}$ is a Hamiltonian used to model both quantum corrals \cite{Aligia,Heller} and leaky quantum graphs \cite{Exner}. Since highly frequency dependent delta potentials are out of reach of the more general techniques in \cite{Galk,GS}, we study the special case where $\Omega=B(0,1)\subset \re^2$ and $V\equiv h^{-\alpha }V_0>0$ with $\alpha\leq 1$. Here $h^{-1}\sim \Re \lambda$ is the frequency. We give sharp bounds on the size of resonance free regions for $\alpha\leq 1$ and the location of bands of resonances when $5/6\leq \alpha\leq 1$. Finally, we give a lower bound on the number of resonances in logarithmic size strips: $-M\log \Re \lambda\leq \Im \lambda \leq 0$.  
\end{abstract}

\maketitle

\section{Introduction}
Scattering by potentials is used in mathematics and physics to study long term behavior of waves in many physical systems  (see for example \cite{Burke},\cite{ZwScat}, \cite{Lax}, and \cite{ZwAMS} the references therein). Examples include the study of acoustics in concert halls, scattering of gravitational waves by black holes and scattering in open microwave cavities. 

Recently, there has been interest in scattering by quantum corrals that are constructed using scanning tunneling microscopes \cite{Aligia, Heller,Crommie} and in scattering by leaky quantum graphs \cite{Exner}. One model used in the theoretical understanding of these systems is a delta function potential on the boundary of a domain $\Omega\subset \re^d$ (see for example \cite{crommie1995waves,Exner}). The papers \cite{Galk,GS} began the rigorous study of scattering by delta function potentials on hypersurfaces. 

In \cite[Theorem 1.4]{GS}, Smith and the author show that solutions to
\begin{equation}
\label{eqn:wave} 
(\partial_t^2-\Delta +\delta_\pO\otimes V)u=0
\end{equation} 
where $\delta_{\pO}$ is the Hausdorff $d-1$ measure on the hypersurface, $\pO\subset\re^d$ ($d$ odd) and $V:L^2(\pO)\to L^2(\pO)$, 
have, for any $K$ a compact subset of $\re^d$, expansions roughly of the form 
\begin{equation}
\label{eqn:expand}
u(t,x) \sim \sum_{\lambda\in \text{Res}}e^{-it\lambda}a_\lambda u_\lambda (x),\quad\quad x\in K\Subset \re^d
\end{equation}
where $\text{Res}$ is the (discrete) set of \emph{scattering resonances}. Thus, the real and (negative) imaginary parts of resonances correspond respectively to the frequency and decay rate of the associated resonant state, $u_\lambda(x).$ Here a \emph{resonant state at $\lambda$}  is a nonzero outgoing solution to 
$$ (-\Delta+\delta_{\pO}\otimes-\lambda^2)u=0.$$ 
The expression \eqref{eqn:expand} is similar to the expansion in terms of eigenvalues that one obtains for the solution of the wave equation on a compact manifold. Hence, for leaky systems, scattering resonances play the role of eigenvalues in the closed setting. Resonance expansions like \eqref{eqn:expand} appear in a wide variety of scattering problems (see for example \cite[Sections 3.2 ,4.6]{ZwScat} \cite{BuZw,TangZw} and the references therein).

As can be seen from \eqref{eqn:expand}, resonances close to the real axis give information about long term behavior of waves. Since the seminal work of Lax--Phillips \cite{Lax} and Vainberg \cite{Vain}, (asymptotically) resonance free regions near the real axis have been used to understand decay of waves. In particular, if there are no resonances, $\lambda$, with $|\Re \lambda|\geq M$ and $\Im \lambda > -\gamma$, then there are only finitely many $\lambda \in \text{Res}$ with $\Im \lambda >-\gamma$. Thus, an expansion of the form \eqref{eqn:expand} implies 
\begin{equation}
\label{eqn:expand2}u(t,x)\sim \sum_{\substack{\lambda \in \text{Res}\\\Im \lambda >-\gamma}}e^{-it\lambda}a_\lambda u_\lambda(x)+\O{}(e^{-t\gamma}),\quad\quad x\in K\Subset \re^d.
\end{equation}
That is, in any compact set, there is a (non-orthogonal) expansion of $u(t,x)$ into time harmonic pieces up to an exponentially decaying error. Moreover, if there are no resonances with $|\Re \lambda|\geq M$ and $\Im \lambda \geq -C\log |\Re \lambda|$, then the error term in \eqref{eqn:expand2} becomes smoother as $t\to \infty$ (see \cite{GaQV} and the references therein.)

While the spectral analysis of $-\Delta+V\otimes \delta_{\pO}$ below applies equally well to the analysis of the Schr\"odinger equation
$$(i\partial_t-\Deltad{\pO})u=0,$$
expansions of the form \eqref{eqn:expand2} generally do not hold. Instead, one must take initial data $u(0)$, $u_t(0)$ concentrated at frequency $\sim\lambda$. Then, under various assumptions on resonances, one obtains weaker versions of \eqref{eqn:expand} (see for example \cite{Bu,LocSmooth1,NakSteZw} \cite[Chapter 7]{ZwScat}).

\begin{figure}
\includegraphics[width=0.5\textwidth]{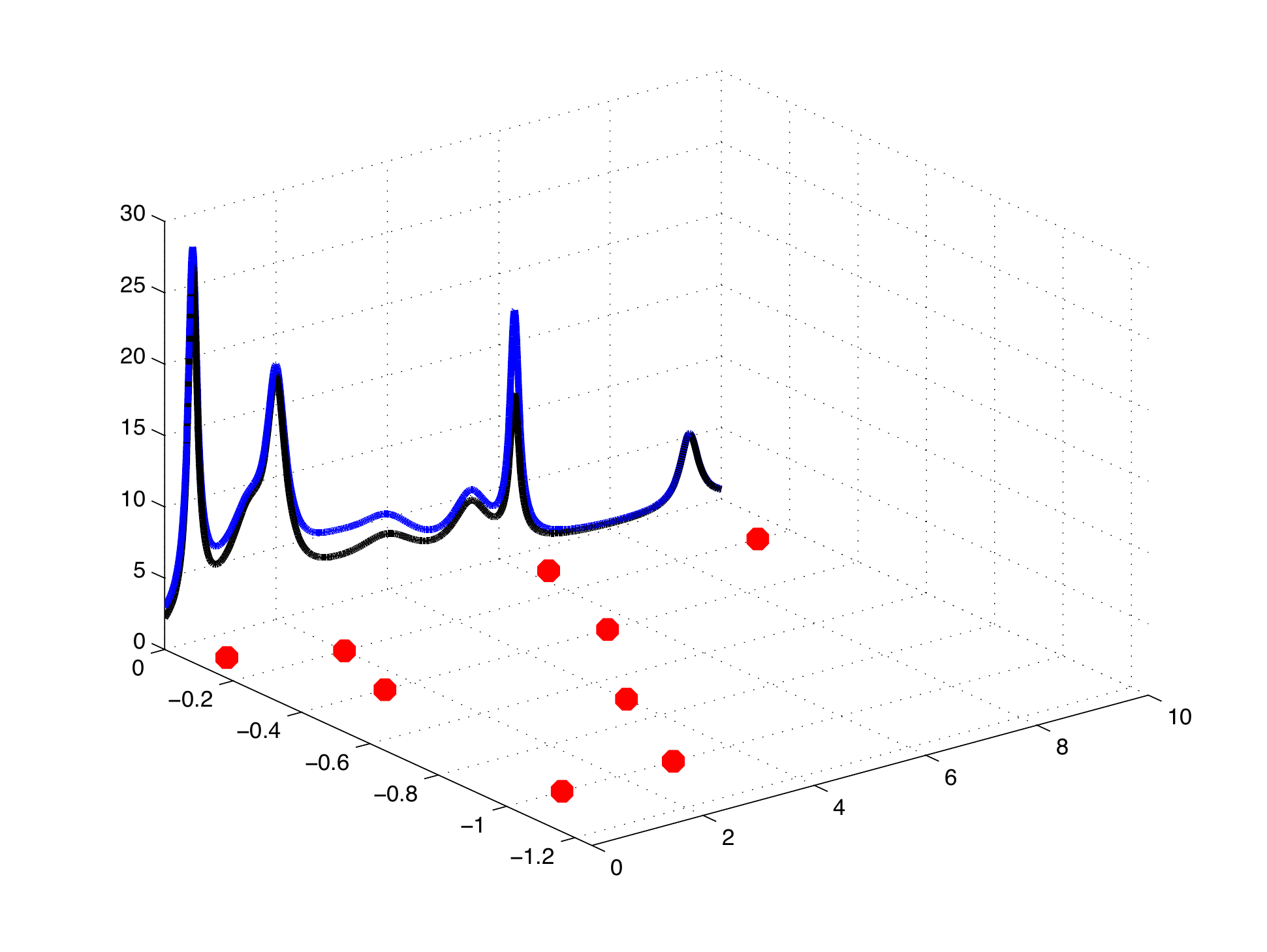}
\caption{We show the power spectrum recovered from waves after a long period in blue. The resonances associated with this power spectrum are shown in red.  A resonance close to the real axis causes a thin and narrow spike, while one further from the real axis causes a broader peak. Thus the bulk of the energy is contained where resonances are close to the real axis. Power spectra like those shown are often used in experiments to recover resonances.}
\end{figure}

Let $-\Deltad{\pO}$ be the unbounded operator
 $
-\Deltad{\pO}:=-\Delta +\delta_{\partial\Omega}\otimes V.$ 
 We assume that $\partial\Omega\subset \re^d$ is a smooth, bounded hypersurface. (For more general assumptions see \cite{Galk, GS}.) We take
$V$ a bounded operator on $L^2(\partial\Omega)$. For
$u\in H^1_{\loc}(\re^d)$, we define 
$(V\otimes\delta_{\partial\Omega})u:=(Vu|_{\partial\Omega})\delta_{\partial\Omega}$. (See Section \ref{sec:formalDefinition} for the formal definition of $-\Deltad{\pO}$.)

\begin{remark}Just as for $V\in L^\infty_{\comp}$, $-\Deltad{\pO}$ has $L^2$ spectrum bounded from below with essential spectrum equal to $[0,\infty)$. However, unlike for $V\in L^\infty_{\comp}$, $-\Delta_{\pO}$ may have embedded eigenvalues (see \cite[Sections 2.1, 7]{GS}).   
\end{remark}

The precise definition of a \emph{scattering resonance} for $-\Deltad{\pO}$ is a pole of the meromorphic continuation  of the resolvent 
$$
R_V(\lambda)=(-\Deltad{\pO}-\lambda^2)^{-1}\,,
$$
from $\Im \lambda \gg 1$. Because we are interested in asymptotically resonance free regions near the real axis, and in particular, $|\Re \lambda|\gg 1$, it is convenient to rescale $\lambda=z/h$ with $h\ll1$ and write 
$$R_V(z/h)=h^2(-h^2\Deltad{\pO}-z^2)^{-1}.$$
Thus, we study the poles of 
\begin{equation}
\label{eqn:rescaledOp}(-h^2\Deltad{\pO}-z^2)^{-1}=(-h^2\Delta +(h \delta_\pO\otimes hV)- z^2)^{-1}.\end{equation}

In typical physical systems such as quantum corrals considered in \cite{Heller} and concert halls, interactions between barriers and waves are frequency dependent. One natural model where this is the case is the quantum point interaction in one dimension. This object is understood using the operator on $L^2(\re)$
\begin{equation}
\label{eqn:quantumPoint}-h^2\Delta +V\otimes\delta(x/h)=-h^2\Delta+h^{-1}(h\delta(x)\otimes hV).
\end{equation}
This type of operator also appears when considering delta function potentials on domains of scale comparable to the frequency of interest
\begin{equation}
\label{eqn:growingDomain}-\Delta+\delta(x-h^{-1})+\delta(x+h^{-1}).
\end{equation}
In particular, rescaling with $hy=x$ results in the operator 
$$-h^2\Delta+\delta((y-1)/h)+\delta((y+1)/h)=-h^2\Delta+h\delta(y-1)+h\delta(y+1).$$
Both \eqref{eqn:quantumPoint} and \eqref{eqn:growingDomain} correspond to letting $V\sim h^{-1}$ in \eqref{eqn:rescaledOp}. Thus, the natural upper bound on $V$ is $\|V\|_{L^2\to L^2}\leq Ch^{-1}$. 

One additional way in which this type of operator appears is in the wave equation
$$(\partial_t^2-\Delta +i\delta_{\pO}\otimes (\la a,\partial_x\ra +a_0\partial_t))u=0.$$
If we formally take the Fourier transform in time, we arrive at 
$$[-\Delta +\lambda \delta_{\pO}\otimes (\la\lambda^{-1} a,\partial_x\ra-ia_0)-\lambda^2]\hat{u}=0$$
and, rescaling $\lambda=z/h$ gives
\begin{equation}\label{eqn:waveDelta}[-h^2\Delta+z(h\delta_{\pO}\otimes (\la z^{-1}a,h\partial_x\ra -ia_0))-z^2]\hat{u}=0.
\end{equation}

To understand how resonances behave for highly frequency dependent potentials, we consider a model potential which depends strongly on frequency $(\sim h^{-1})$. In particular, we consider $-\Deltad{\pO}$ when $\Omega=B(0,1)\subset \re^2$ and $V\equiv h^{-\alpha} V_0$ for $\alpha \leq 1$, and $V_0> 0$ is a constant independent of $h$. 

Some progress has been made toward understanding the distribution of resonances for delta potentials depending on frequency. In \cite{GS}, Smith and the author demonstrate the existence of a logarithmic resonance free region for a very general class of $\Omega$. The results imply the existence  of logarithmic resonance free regions in our case when $\alpha<2/3$. In \cite{Galk}, the present author gives sharp bounds on the size of the resonance free region when $V\in h^{-\alpha}C^\infty(\partial\Omega)$ with $ \alpha < 2/3$ and $\Omega$ a smooth strictly convex domain. Because $\alpha<2/3$, we think of the potentials considered in \cite{Galk} and \cite{GS} as depending mildly on frequency. 

However, the quantum point interaction \eqref{eqn:quantumPoint}, large domain \eqref{eqn:growingDomain}, and first order delta potential \eqref{eqn:waveDelta} correspond to a potential which depends strongly on frequency ($\alpha=1$). The range $2/3\leq \alpha\leq 1$ is out of reach using the techniques from \cite{Galk,GS} because of a complication in the microlocal analysis near trajectories tangent to the boundary of $\Omega$, \emph{glancing trajectories}.

Denote the set of rescaled resonances for $-\Deltad{\partial\Omega}$ by
\begin{equation}
\label{def:lambdaCircle}
\begin{gathered} \resd:=\{z\in [1-ch^{3/4},1+ch^{3/4}]+i[-Mh\log h^{-1},0]: z/h\text{ is a resonance of } -\Deltad{\partial\Omega}\}.
\end{gathered}
\end{equation}

Our first theorem proves the existence of resonance free regions for $\alpha \leq 1$ and bands of resonance free regions for $1\geq \alpha\geq 5/6$.
\begin{theorem}
\label{thm:resFreeCircle}
Let $\Omega=B(0,1)\subset \re^2$ and $V\equiv h^{-\alpha}V_0> 0$. Then for all $\e>0$, $M>0$ there exists $h_{\e,M}>0$ such that for $0<h<h_{\e,M}$,
\begin{equation*}
\Lambda(h)\subset\left\{\begin{aligned}\left\{-\Im z\geq \frac{1-\alpha}{2}h\log h^{-1} -\frac{h}{2}\log \frac{V_0}{2}-\e h\right\}\cup\left\{-\Im z\geq Mh\log h^{-1}\right\}&\quad\alpha<5/6\\[5pt]
\left\{|-\Im zh^{2/3-2\alpha}-C_{V_0,N}|<\e\right\}\cup\left\{-\Im z\geq Mh^{2\alpha-2/3}\right\}&\quad5/6\leq\alpha\leq 1\end{aligned}\right.
\end{equation*}
where 
$$C_{V_0,N}:=\frac{\sqrt[3]{2}}{8\pi^2V_0^2|A_-(-\zeta_N)^3Ai'(-\zeta_N)|}$$
and $-\zeta_N$ is the $N^\text{th}$ zero of the Airy function $Ai(s)$.
\end{theorem}
\noindent See Figure \ref{f:thmPic} for a pictorial representation of the results of Theorem \ref{thm:resFreeCircle}.

The next theorem shows that Theorem \ref{thm:resFreeCircle} is optimal.
\begin{theorem}
\label{thm:existenceCircle}
For all $N>0$, there exists $h_0>0$ such that for $h<h_0$, there exist $z(h)\in \Lambda$ with 
$$-\Im z(h)= \begin{cases}\frac{1-\alpha}{2}h\log h^{-1}-\frac{h}{2}\log \frac{V_0}{2}+\O{}(h^{7/4})&\alpha < 1\\
\frac{h}{4}\log\left(1+\frac{4}{V_0^2}\right)+\O{}(h^{7/4})&\alpha =1\\
C_{V_0,N}h^{2\alpha-2/3}+\O{}(h^{3\alpha-4/3})&2/3<\alpha \leq 1\end{cases}
$$
\end{theorem}

\begin{figure}
\includegraphics{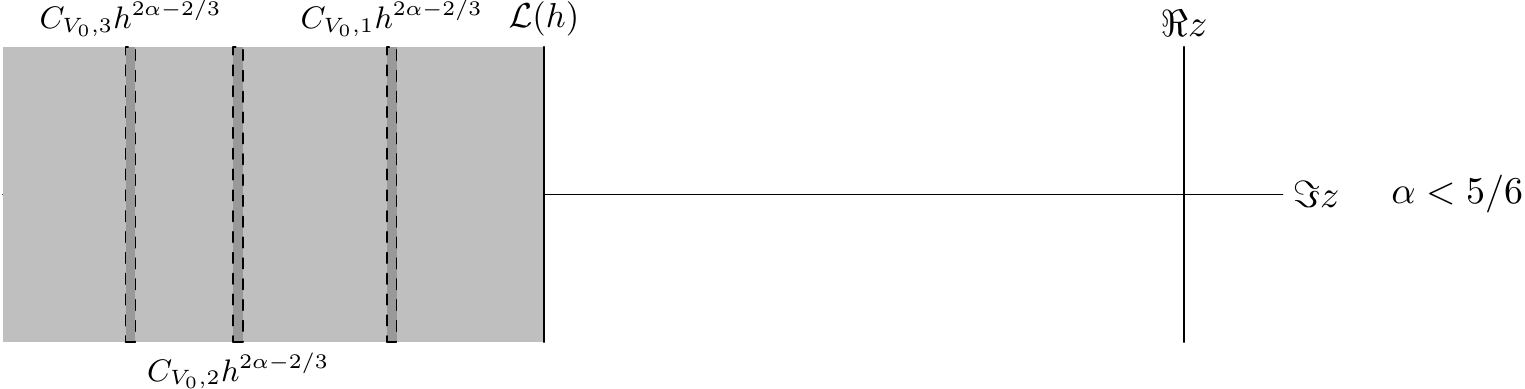}\\[20pt]
\includegraphics{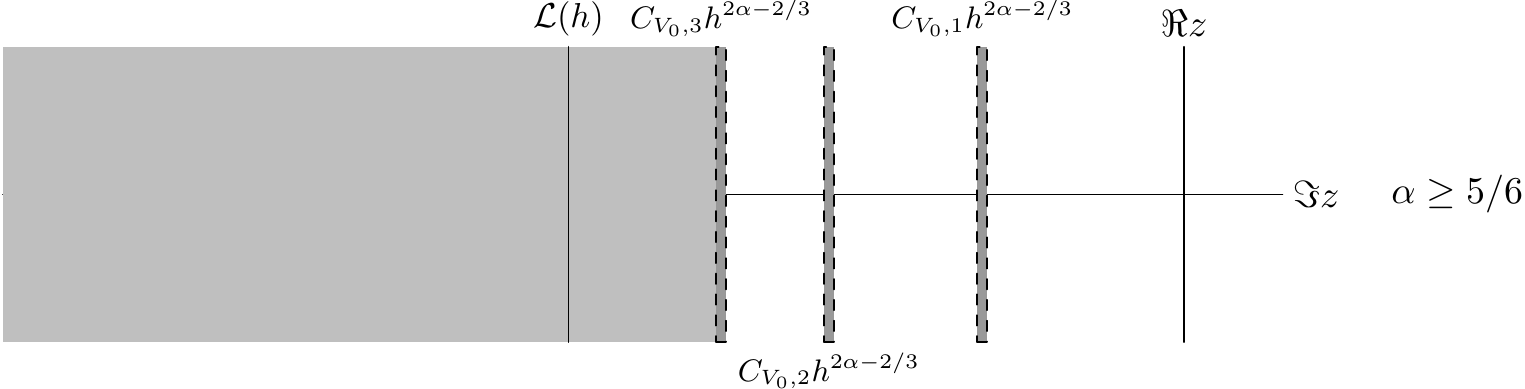}
\caption{\label{f:thmPic} This figure is a schematic of the results of Theorems \ref{thm:resFreeCircle} and \ref{thm:existenceCircle} for $\alpha<5/6$ on the top and $\alpha\geq 5/6$ on the bottom. Resonances lie in either the dark gray bands or the light gray shaded region, but not in the white regions. Here, $\mc{L}(h):=\frac{1-\alpha}{2}h\log h^{-1}-\frac{h}{2}\log \frac{V_0}{2}.$ Note that we show only three of the bands given by Theorem \ref{thm:resFreeCircle}.  See also Figures \ref{fig:circleRes}, \ref{fig:numericalResonanes}, \ref{fig:numRes2}, and \ref{fig:resBands} for numerically computed resonances.}
\end{figure}

The proof of Theorems \ref{thm:resFreeCircle} and \ref{thm:existenceCircle} show that when $\alpha<5/6$ the resonances closest to the real axis come from modes concentrating microlocally away from glancing trajectories, while those for $\alpha\geq 5/6$ come from modes concentrating near glancing. Thus, the theorems show that glancing modes decay slower than non-glancing modes for $\alpha\geq 5/6$ while the opposite is true for $\alpha<5/6$ and gives a quantitative rate of decay for each type of mode.

\begin{remark}
When $B(0,1)$ is replaced by $B(0,R)$ we can rescale to find that the resonance free region for $\Omega=B(0,R)$ and $\alpha\geq 5/6$ is given by
$-\Im z\geq
(C_{R^{2/3} V_0}-\e)h^{2\alpha -2/3}.$
Hence the imaginary part of resonances from glancing modes scale as $\kappa^{4/3}$ where $\kappa$ is the curvature.
\end{remark}

\begin{figure}
\centering
\includegraphics[width=4.5in]{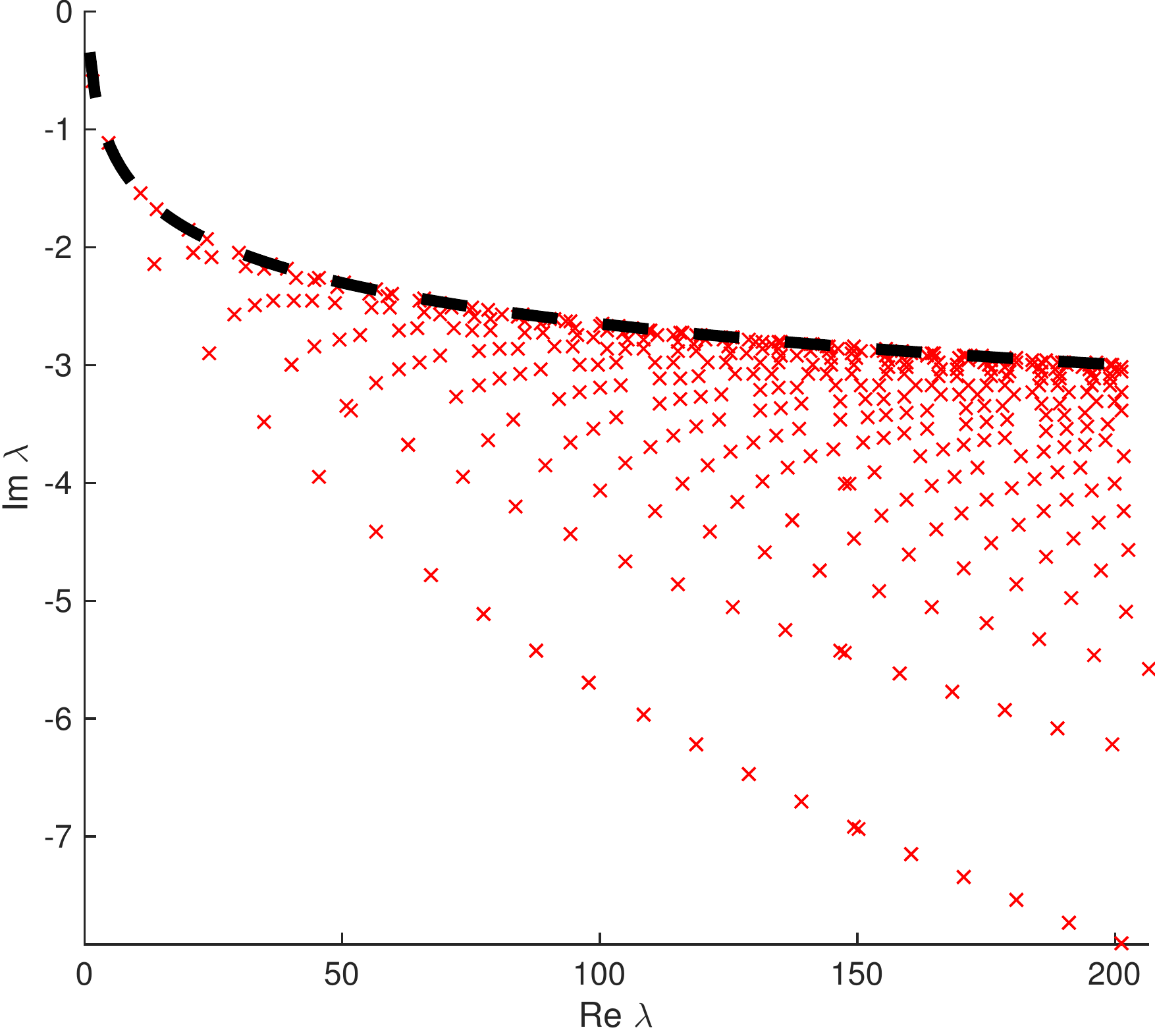}
\caption[Resonances computed for a circle]{\label{fig:circleRes}When $\Omega=B(0,1)\subset \re^2$, the boundary values of resonance states can be expanded in a Fourier series $\sum a_n e^{inx}.$ We show the resonances for $V\equiv 1$ corresponding roughly to $n\leq 200$ and $\lambda \leq 200$. The dashed line shows the bound given by Theorem \ref{thm:resFreeCircle}.}
\end{figure}

We also give a lower bound on the number of resonances. \begin{theorem}
\label{thm:lowerBound}
For $M$ large enough, there exists $c>0$ such that 
$$\#\{z\in [1-\e,1+\e]+i[-Mh\log h^{-1},0]:z/h\text{ is a resonance of } -\Deltad{\pO}\}\geq ch^{-2}.$$
\end{theorem}
\begin{remark}We have an upper bound of the form $Ch^{-2}$  by  \cite{SjoDist}, \cite{Vod1}, \cite{Vod2}, and \cite{Vod3} together with \cite[Lemma 7.1]{GS}.\end{remark}

\begin{figure}
\centering
\begin{subfigure}[b]{.45\textwidth}
\includegraphics[width=\textwidth]{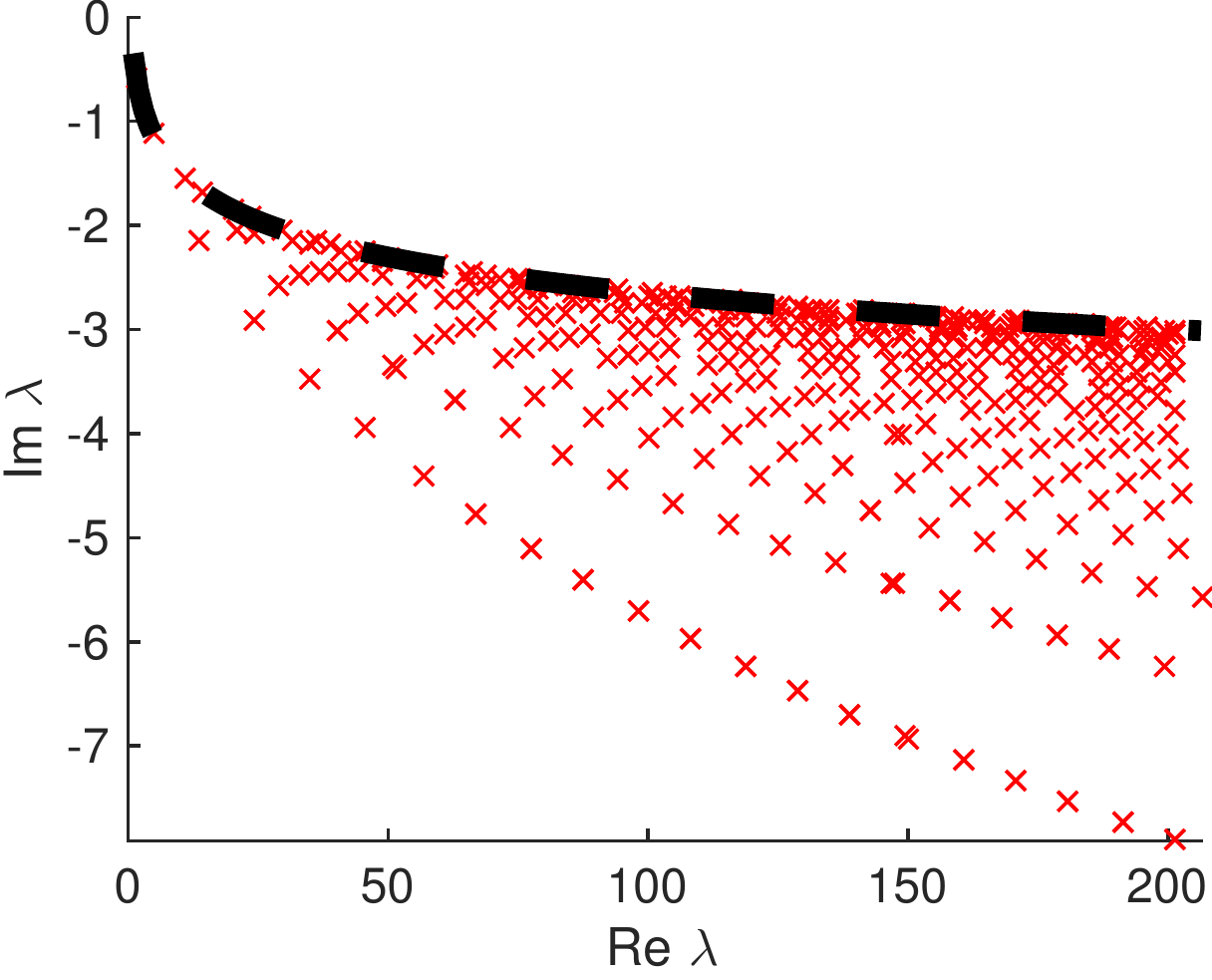}
\caption{$\alpha=0$}
\end{subfigure}
\begin{subfigure}[b]{.45\textwidth}
\includegraphics[width=\textwidth]{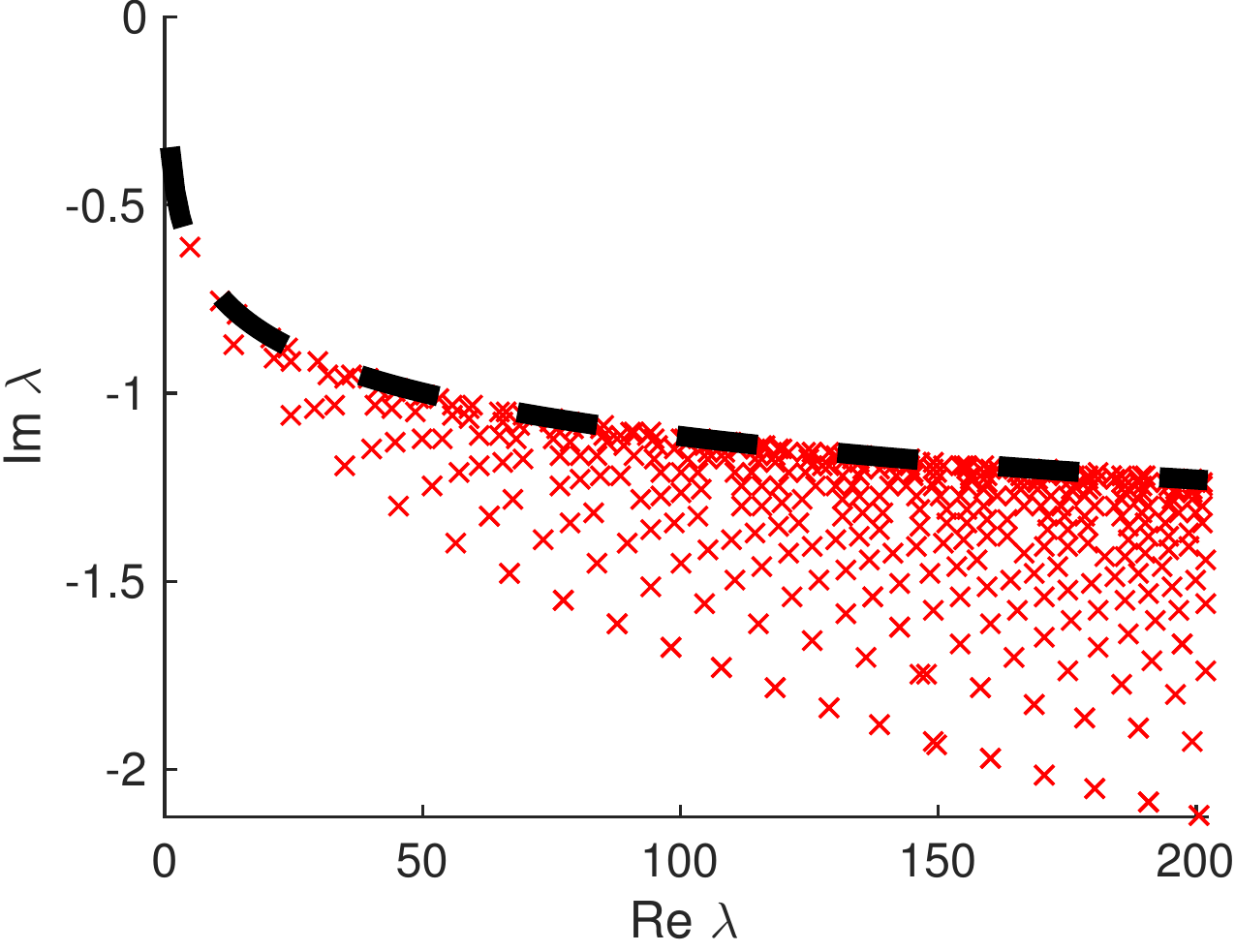}
\caption{$\alpha=\frac{2}{3}$}
\end{subfigure}

\begin{subfigure}[b]{.45\textwidth}
\includegraphics[width=\textwidth]{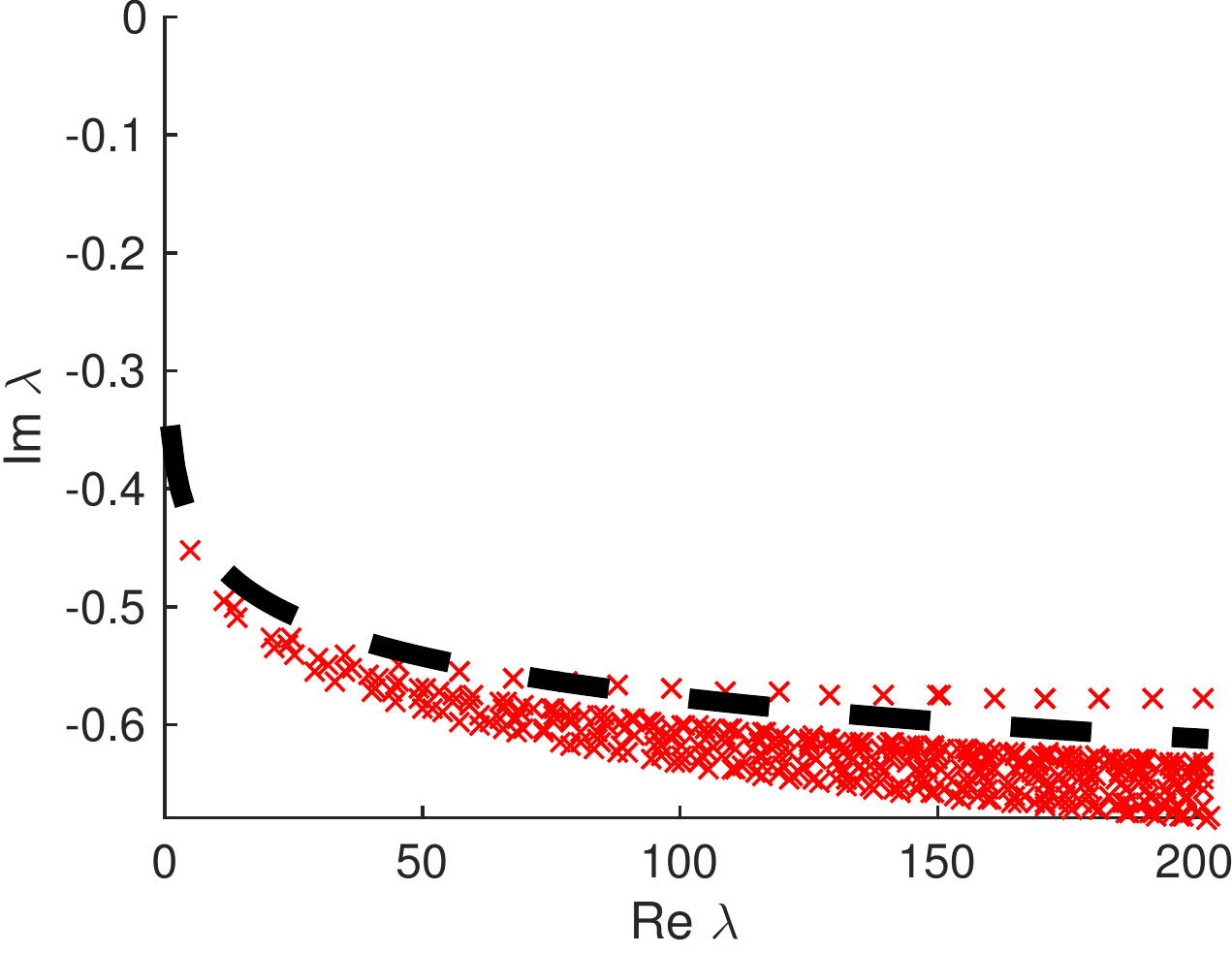}
\caption{$\alpha=0.9$}
\end{subfigure}
\begin{subfigure}[b]{.45\textwidth}
\includegraphics[width=\textwidth]{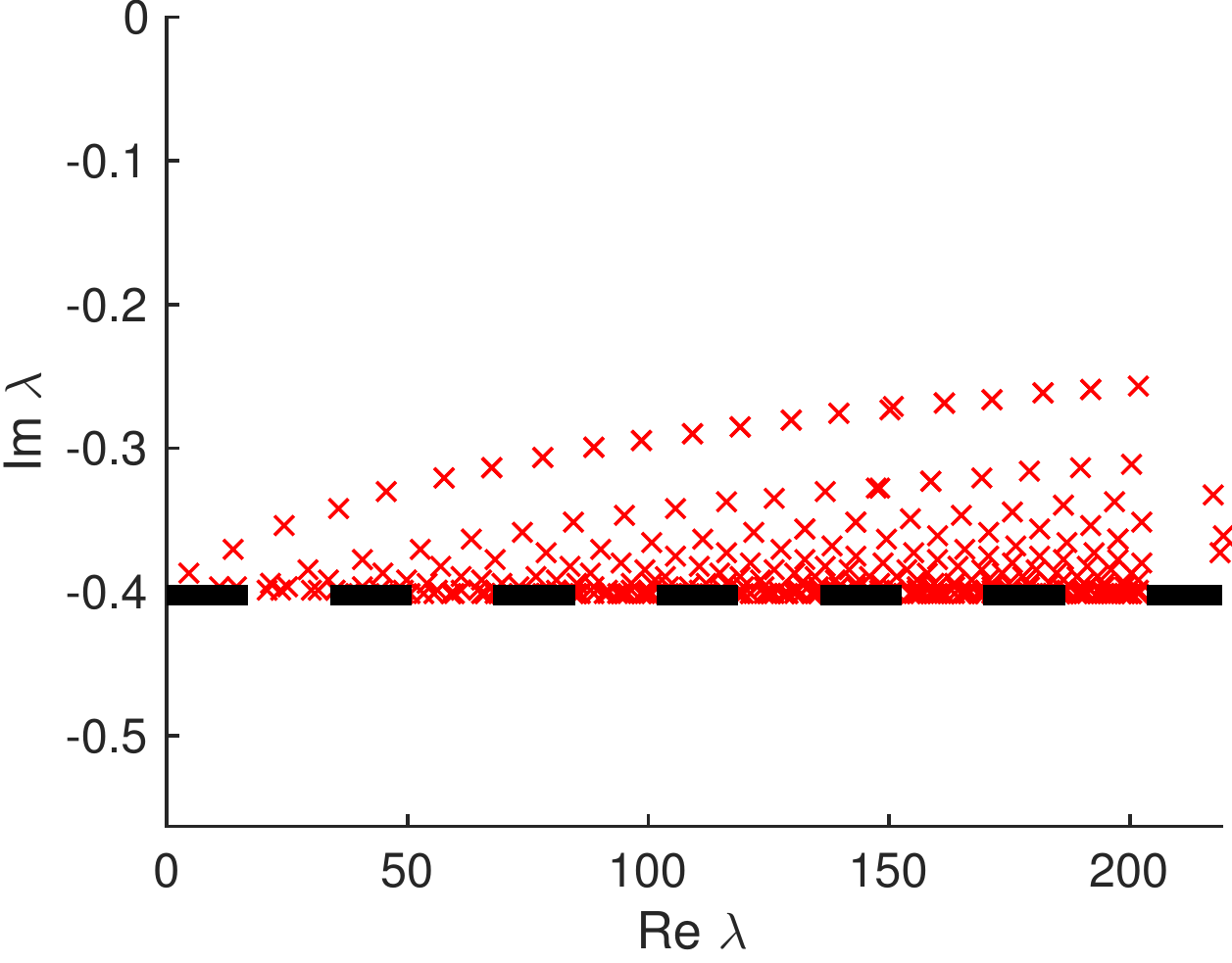}
\caption{$\alpha=1$}
\end{subfigure}
\caption[Numerical computation of resonances for the disk with highly frequency dependent potential small $\Re \lambda$]{We show resonances for the circle with $\Re \lambda\sim 10^2$, $V_0=1$ and several $\alpha$. The plots show $\Im \lambda $ v $\Re \lambda $ in each case. The dashed line shows the bound coming from non-glancing modes. It is difficult to see the transition at $\alpha=5/6$ from logarithmic resonance free regions to polynomial size resonance free regions because the change does not happen until $\Re \lambda \sim 10^6$ (see Figure \ref{fig:numRes2}).}
\label{fig:numericalResonanes}
\end{figure}

\begin{figure}
\centering
\begin{subfigure}[b]{.45\textwidth}
\includegraphics[width=\textwidth]{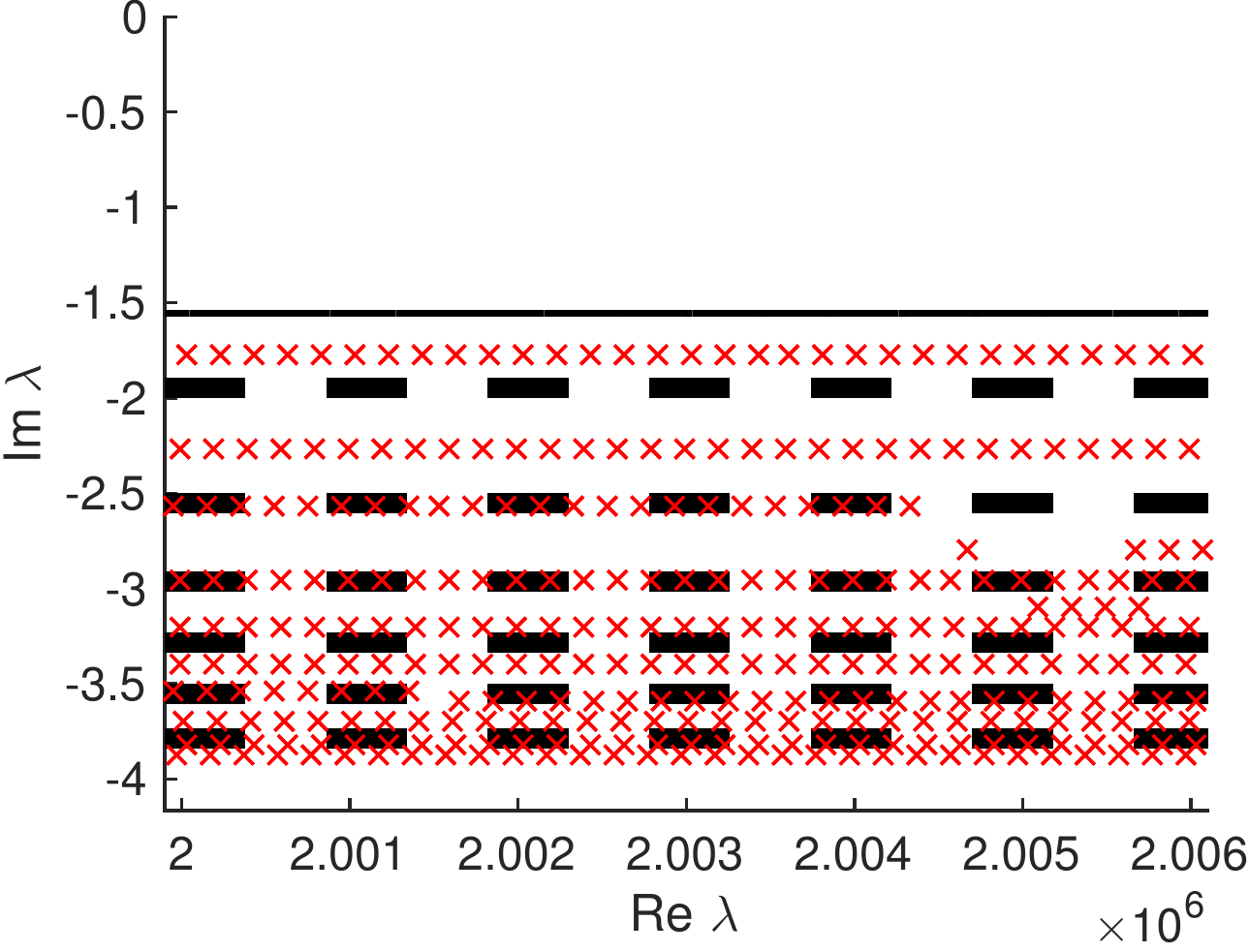}
\caption{$\alpha=\frac{5}{6}$}
\end{subfigure}
\begin{subfigure}[b]{.45\textwidth}
\includegraphics[width=\textwidth]{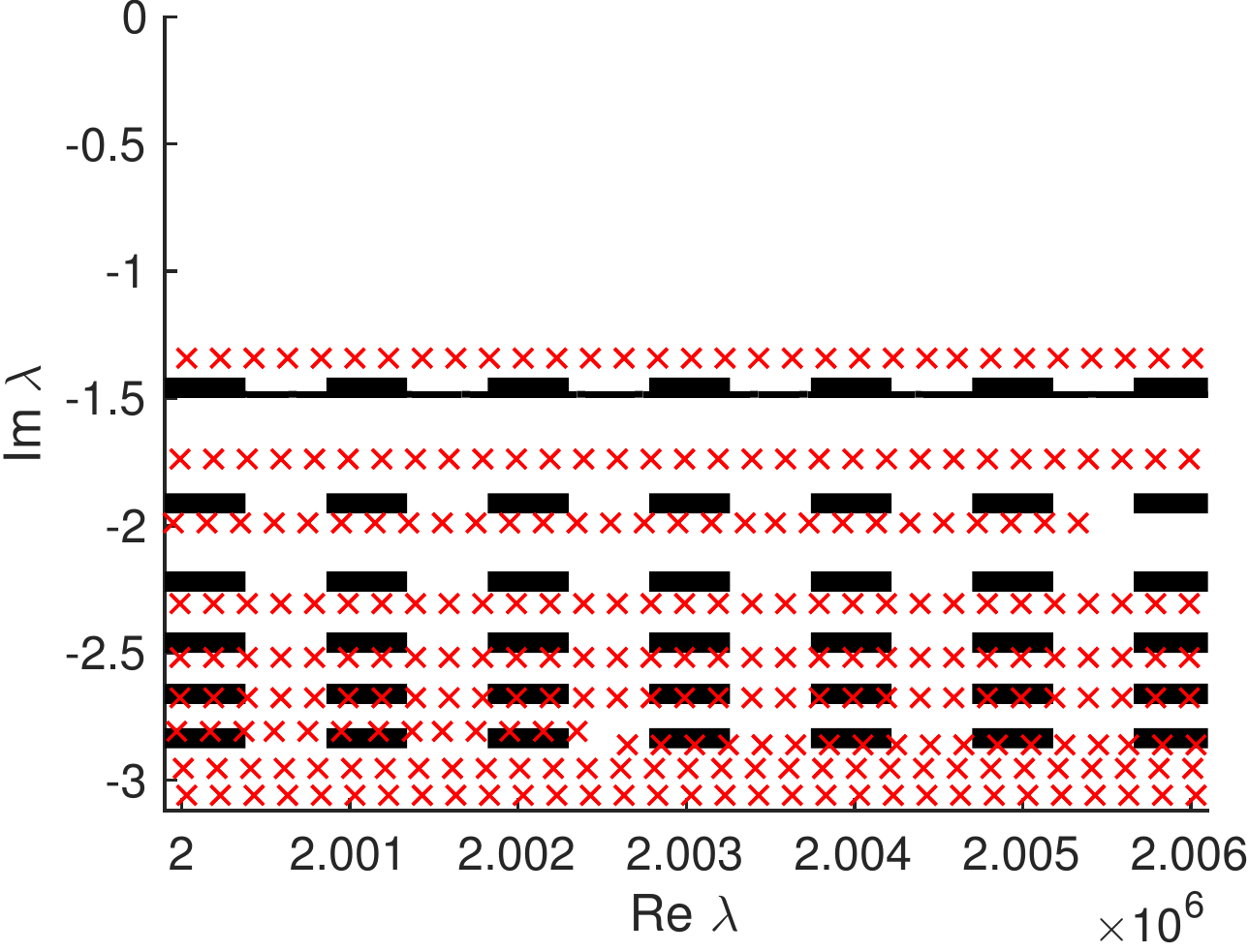}
\caption{$\alpha=0.8433$}
\end{subfigure}
\begin{subfigure}[b]{.45\textwidth}
\includegraphics[width=\textwidth]{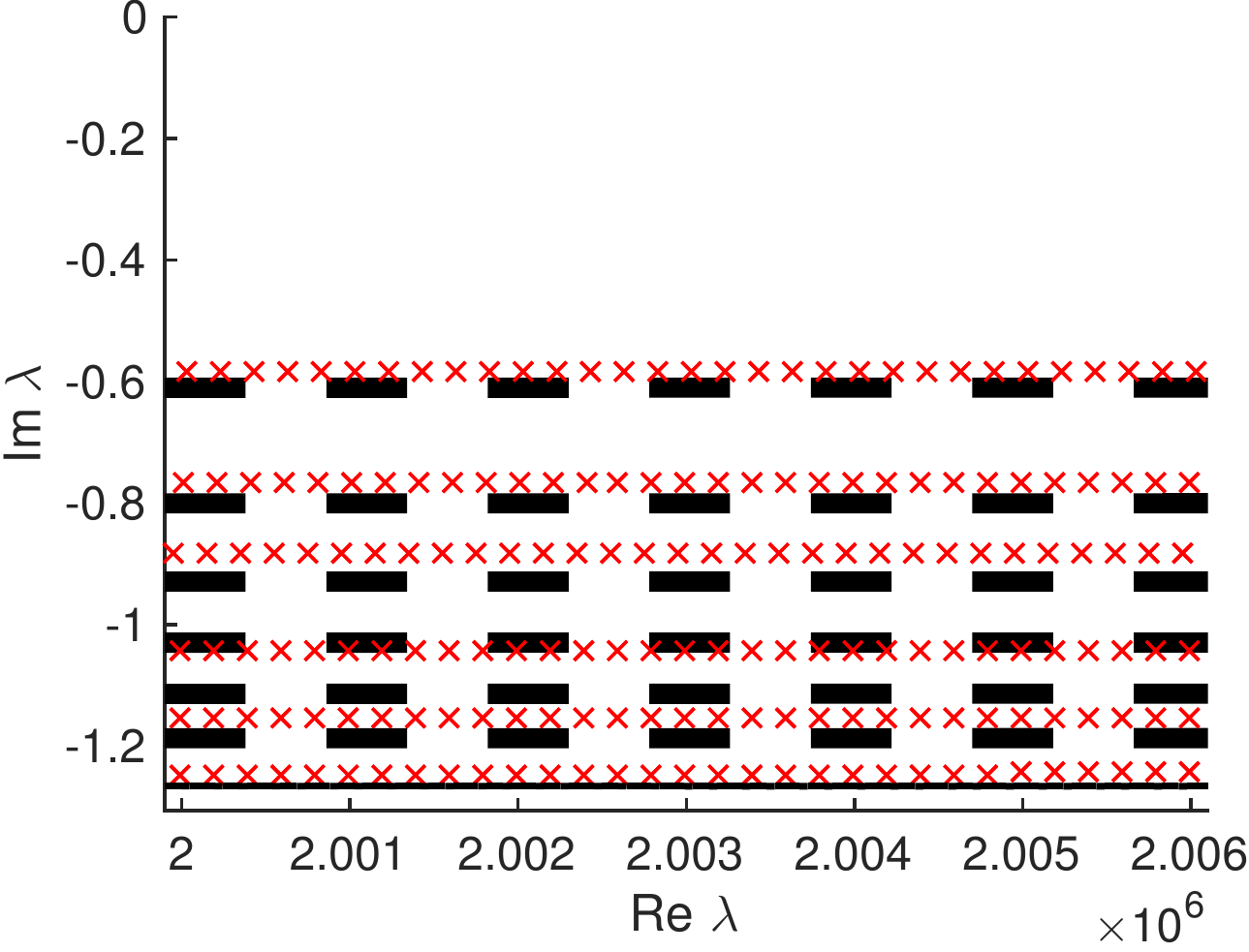}
\caption{$\alpha=0.8733$}
\end{subfigure}
\begin{subfigure}[b]{.45\textwidth}
\includegraphics[width=\textwidth]{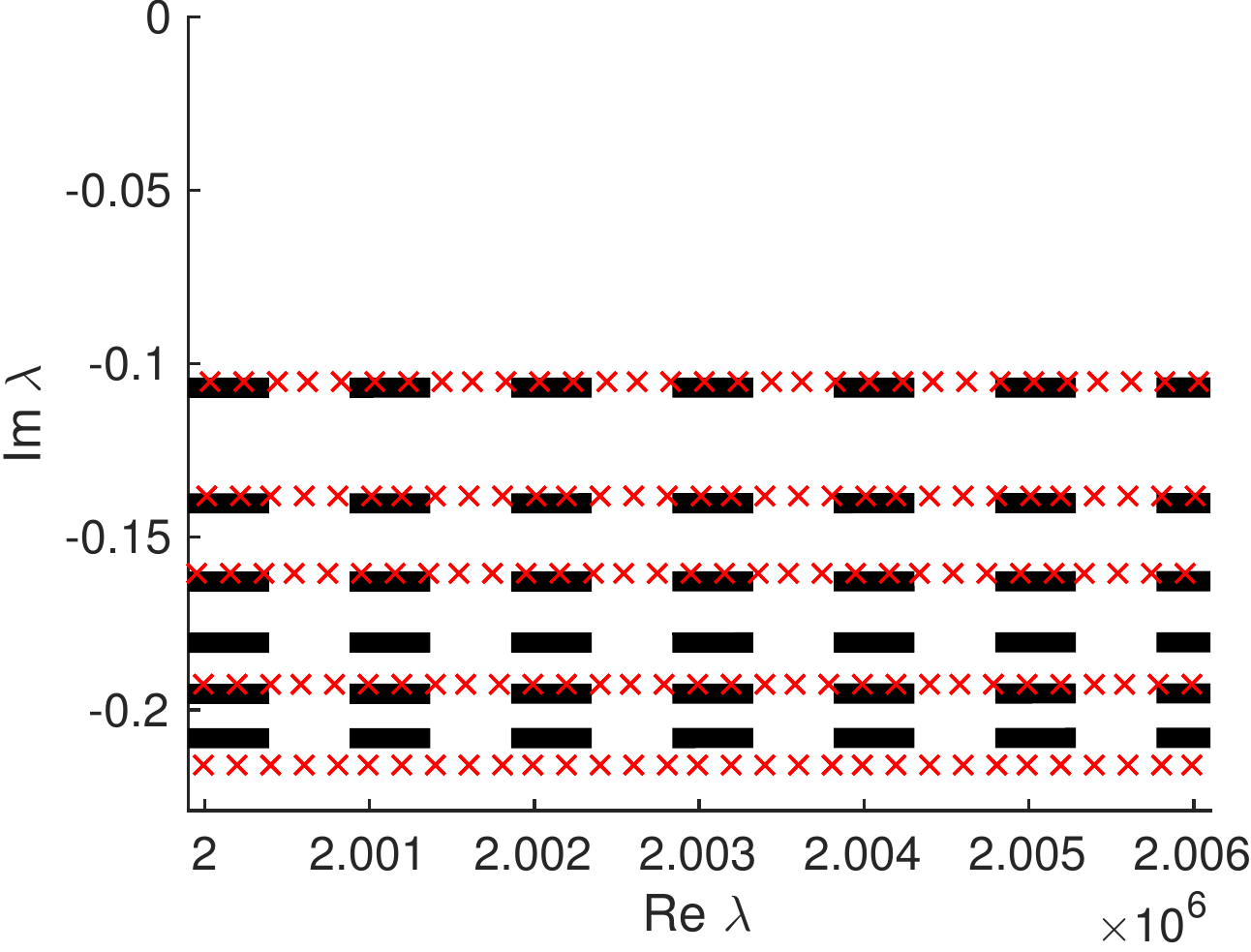}
\caption{$\alpha=0.9333$}
\end{subfigure}
\caption[Numerical computation of resonances for the disk with highly frequency dependent potential large $\Re \lambda$]{We show resonances for the circle with $\Re \lambda\sim 2\times 10^6$, $V_0=1$ and several $\alpha$. The plots show $\Im \lambda $ vs. $\Re \lambda $ in each case. The solid black line shows the (logarithmic) bound for resonances coming from non-glancing trajectories and the dashed lines show the first few (polynomial) bands of resonances from near glancing trajectories. Since the dashed red line is above the black lines at $\alpha=5/6$, it is necessary to go to still larger $\Re \lambda$ to see the transition. However, when $\alpha>5/6$, we start to see better agreement with the bands of resonances predicted in Theorem \ref{thm:resFreeCircle}. Note that the logarithmic line is off of the graph when $\alpha=0.9333$.}
\label{fig:numRes2}
\end{figure}

\begin{figure}
\centering
\includegraphics[width=.75\textwidth]{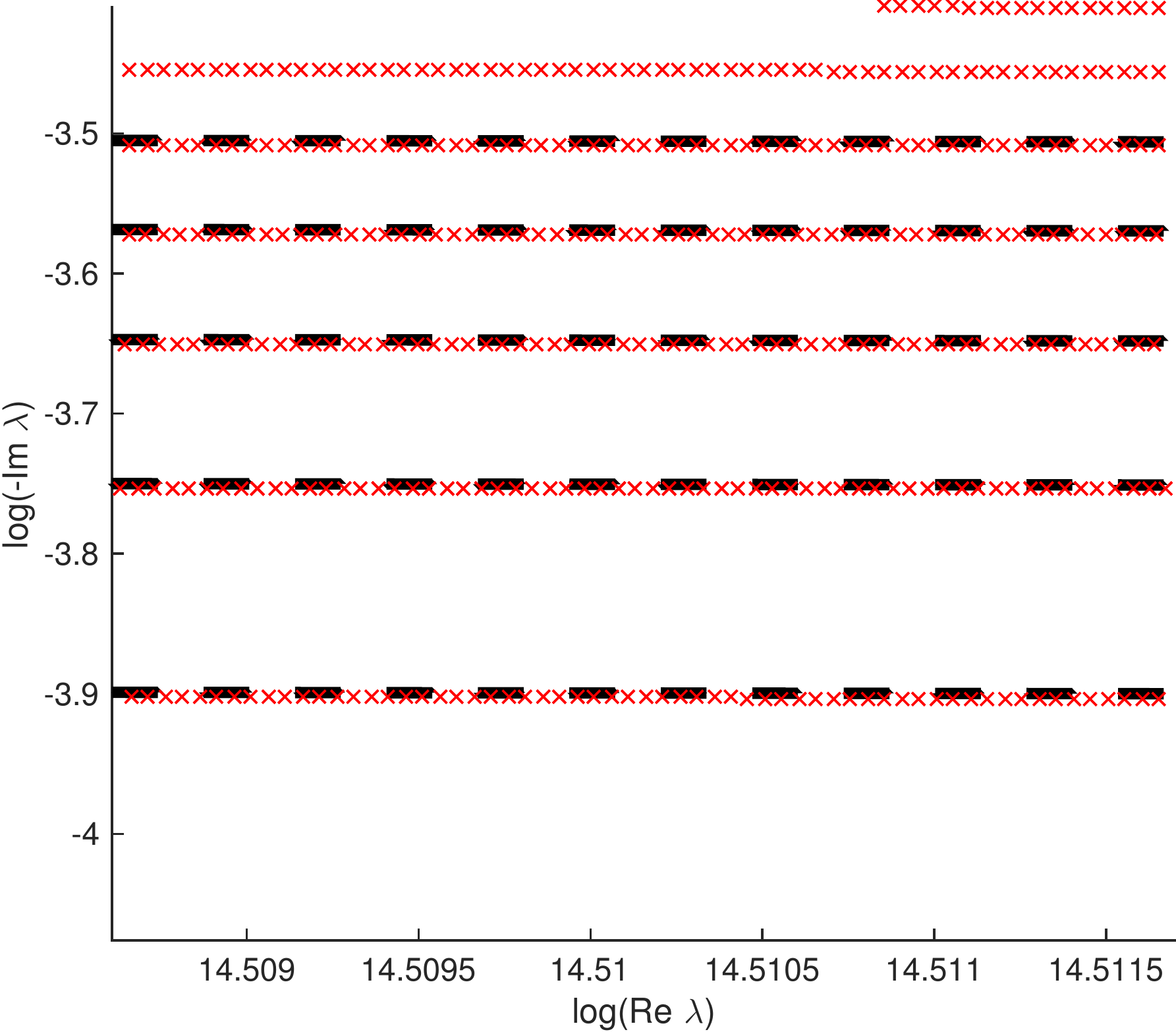}
\caption[Numerical computation of resonance bands for the disk (log-log plot)]{We show a plot of $\log(\Re \lambda)$ vs. $\log(-\Im \lambda)$ for $\Re \lambda \sim 2\times 10^6$ when $\alpha=1$ and $V_0\equiv 1$. The bands predicted by Theorem \ref{thm:resFreeCircle} are shown by the dashed lines. \label{fig:resBands}}
\end{figure}

The paper is organized as follows. In Section \ref{sec:preliminaries} we recall some asymptotics of Bessel and Airy functions, introduce the formal definition of $\Delta_{V,\Omega}$ and use the results of \cite{GS} to reduce the problem to finding solutions of a transcendental equation. In Section \ref{sec:resFreeCircle}, we demonstrate the existence of the various resonance free regions in Theorem \ref{thm:resFreeCircle}. Finally, in Section \ref{sec:existence}, we show the existence of the resonances in Theorem \ref{thm:existenceCircle} and prove Theorem \ref{thm:lowerBound}.

\noindent
{\sc Acknowledgements.} The author would like to thank Maciej Zworski for invaluable guidance and discussions.  The author is grateful to the National Science Foundation for support under the National Science Foundation Graduate Research Fellowship Grant No. DGE 1106400 and grant DMS-1201417 as well as under the Mathematical Sciences Postdoctoral Research Fellowship DMS-1502661.

\section{Preliminaries}
\label{sec:preliminaries}

\subsection{Asymptotics for Airy and Bessel Functions}
We collect here some properties of the Airy and Bessel functions that are used in the analysis of case case of the unit disk. These formulae can be found in, for example \cite[Chapter 9,10]{NIST}.

Recall that the Bessel of order $n$ functions are solutions to 
$$z^2y''+zy'+(z^2-n^2)y=0.$$
We consider the two independent solutions $H_n^{(1)}(z)$ and $J_n(z)$. The Wronskian of the two solutions is given by 
\begin{equation} 
\label{eqn:wrongskBessel}W(J_n,H_n^{(1)})=J_n{H_n^{(1)}}'(z)-J_n'H_n^{(1)}(z)=\frac{2i}{\pi z}\end{equation}  

We now record some asymptotic properties of Bessel functions. Consider $n$ fixed and $z\to \infty$ 
\begin{align}
J_n(z)&=\left(\frac{1}{2\pi z}\right)^{1/2}\left( e^{i(z-\frac{n}{2}\pi -\frac{1}{4}\pi)}+e^{-i(z-\frac{n}{2}\pi -\frac{1}{4}\pi )}+\O{}(|z|^{-1}e^{|\Im z|})\right)\nonumber\\
H_n^{(1)}(z)&=\left(\frac{2}{\pi z}\right)^{1/2}\left(e^{i(z-\frac{n}{2}\pi -\frac{1}{4}\pi )}+\O{}(|z|^{-1}e^{|\Im z|})\right)\nonumber\\
J_n(z)H_n(z)&=\frac{1}{\pi z}\left(e^{i(2z-n\pi -\frac{1}{2}\pi)}+1+\O{}(|z|^{-1}e^{2|\Im z|})\right)\label{eqn:normalAsymptotics}
\end{align}

Next, we record asymptotics that are uniform in $n$ and $z$ as $n\to \infty$. Let $\zeta=\zeta(z)$ be the unique smooth solution on $0<z<\infty$ to 
\begin{equation}\label{eqn:diffEqzeta}\left(\frac{d\zeta}{dz}\right)^2=\frac{1-z^2}{\zeta z^2}\end{equation} 
 with 
 \begin{gather*} 
 \lim_{z\to 0} \zeta= \infty,\quad\quad \lim_{z\to 1}\zeta= 0,\quad \quad \lim_{z\to \infty}\zeta=-\infty.
 \end{gather*}
  Then
 \begin{align}
 \frac{2}{3}(-\zeta)^{3/2}&=\sqrt{z^2-1}-\asec(z)&1<z<\infty \label{eqn:zetaAsymptotic}\\
 \frac{2}{3}(\zeta)^{3/2}&=\log \left(\frac{1+\sqrt{1-z^2}}{z}\right)-\sqrt{1-z^2}&0<z<1\nonumber\\
\frac{1-z^2}{\zeta z^2}&\to \sqrt[3]{2}&z\to 0\label{eqn:zetaLimit}
 \end{align}
 Let $$Ai(s)=\frac{1}{2\pi}\int _{-\infty }^{\infty }e^{i(\frac{1}{3}t^3+st)}dt$$ 
 for $s\in \re$
be the Airy function solving 
$Ai''(z)-zAi(z)=0.$
Then, $A_-(z)=Ai(e^{2\pi i/3}z)$ is another solution of the Airy equation. 

For $z$ fixed as $n\to \infty$
\begin{align}
J_n(nz)&= \left(\frac{4\zeta}{1-z^2}\right)^{1/4}\left(\frac{Ai(n^{2/3}\zeta)}{n^{1/3}}+\O{}(Ei(5/3,7/3))\right)\nonumber\\
H_n^{(1)}(nz)&=2e^{-\pi i/3}\left(\frac{4\zeta}{1-z^2}\right)^{1/4}\left(\frac{A_-(n^{2/3}\zeta)}{n^{1/3}}+\O{}(E_-(5/3,7/3))\right)\nonumber\\
J_n(nz)H_n^{(1)}(nz)&=2e^{-\pi i/3}\left(\frac{4\zeta}{1-z^2}\right)^{1/2}\left(\frac{Ai(n^{2/3}\zeta)A_-(n^{2/3}\zeta)}{n^{2/3}}+\O{}(Ei_-(\tfrac{8}{3},2,\tfrac{10}{3}))\right)\label{eqn:deltaMainAsymptotics}
\end{align}
where 
\begin{gather*} E_-(\alpha,\beta)=|A_-'(n^{2/3}\zeta)|n^{-\alpha}+|A_-(n^{2/3}\zeta)|n^{-\beta}\\
Ei(\alpha,\beta)=|Ai'(n^{2/3}\zeta)|n^{-\alpha}+|Ai(n^{-2/3}\zeta)|n^{-\beta}\\
Ei_-(\alpha,\beta,\gamma)=|AiA_-|(n^{2/3}\zeta)n^{-\alpha}+(|Ai'A_-|+|AiA_-'|)(n^{2/3}\zeta)n^{-\beta}+|Ai'A_-'(n^{2/3}\zeta)|n^{-\gamma}
\end{gather*}

We now record some facts about the Airy functions $Ai$ and $A_-$. The Wronskian of these two solutions is given by 
\begin{equation}W(A_i,A_-)=AiA_-'(z)-Ai'A_-(z)=\frac{e^{-\pi i/6}}{2\pi}.\label{eqn:WronskianAiry}\end{equation}
Furthermore, for $s\in \re$,
$Ai(s)=e^{-\pi i/3}A_-(s)+e^{\pi i/3}\overline{A_-(s)}$
and hence
\begin{equation} \label{eqn:ImAMinus} 
\Im (e^{-5\pi i/6}A_-(s))=-\frac{Ai(s)}{2}
\end{equation}

The zeros of $Ai(z)$ all lie on $(-\infty,0]$. We use the notation $-\zeta_k$  to denote the $k^{\text{th}}$ zero of $Ai$. 

Finally, we record asymptotics for Airy functions as $z\to \infty$ in the sector $|\Arg z|<\pi/3-\delta$. Many of these asymptotic formulae hold in larger regions, but we restrict our attention to this sector. Let $\eta=2/3z^{3/2}$ where we take the principal branch of the square root. Then
\begin{equation}
\label{eqn:airyAsymps}
\begin{gathered}
A_-(z)=\frac{e^{-\pi i/6}e^{\eta}}{2\sqrt{\pi}z^{1/4}}(1+\O{}(|z|^{-3/2})),\quad\quad \quad\quad A_-(-z)=\frac{e^{\pi i/12}e^{i\eta}}{2\sqrt{\pi }z^{1/4}}\\
Ai(z)=\frac{z^{-1/4}e^{-\eta}}{2\sqrt{\pi}}(1+\O{}(|z|^{-3/2})),\\ Ai(-z)=\frac{z^{-1/4}}{2\sqrt{\pi }}\left(e^{i\eta-i\pi /4}+e^{-i\eta +i\pi /4}+\O{}(|z|^{-3/2}e^{|\Im \eta|})\right) 
\end{gathered}
\end{equation}
\begin{align}
Ai(z)A_-(z)&=\frac{1}{4\pi z^{1/2}}(1+\O{}(|z|^{-3/2})) \label{eqn:AiryAsympPos},&
Ai(-z)A_-(-z)&=\frac{e^{\pi i/3}}{4\pi z^{1/2}}\left(1-ie^{2i\eta}+\O{}(|z|^{-3/2}e^{2|\Im\eta|})\right)
\end{align}

\subsection{Preliminaries on $-\Deltad{\pO}.$}
\label{sec:formalDefinition}
We define the operator $-\Deltad{\pO}$ using the symmetric, densely defined quadratic form 
$$
Q_{V,\pO}(u,w):=\la \nabla u,\nabla w\ra_{L^2(\re^d)} + \la V\gamma u,\gamma w\ra_{L^2(\pO)}\,
$$
with domain $H^1(\re^d)\subset L^2(\re^d)$. Here, $\gamma:H^s(\re^d)\to H^{s-1/2}(\partial\Omega)$ denotes the restriction map. In \cite{GS}, the authors show that $z\in \Lambda$ if and only if there is a nontrivial $z$-outgoing solution to 
\begin{equation} 
\label{eqn:outgoingSolution}
(-h^{2}\Deltad{\pO}-z^2)u=0
\end{equation}
where by $z$-outgoing we mean that there exists $R>0$ such that for $|x|>R$, 
$u=R_0(z/h)g$
for some compactly supported distribution $g$ where $R_0(\lambda)$ is the meromorphic continuation of the \emph{free resolvent}, 
$$R_0(\lambda):=(-\Delta- \lambda^2)^{-1}\,,\quad \quad \Im \lambda \gg1$$
to $\mathbb{C}$ if $d$ is odd and to the logarithmic cover of $\mathbb{C}\setminus\{0\}$ if $d$ is even.

Furthermore, the authors show that if $V:H^{1/2}(\partial\Omega)\to H^{1/2}(\partial\Omega)$ then solving \eqref{eqn:outgoingSolution} is equivalent to solving the following transmission problem
\begin{equation}
\label{eqn:main}
\begin{cases}(-h^2\Delta -z^2)u_1=0&\text{ in }\,\Omega\\
(-h^2\Delta -z^2)u_2=0 & \text{ in }\,\re^d\setminus \overline{\Omega}\\
u_1=u_2&\text{ on }\,\partial \Omega\\
\partial_\nu u_1+\partial_{\nu '}u_2+V\gamma u_1=0&\text{ on }\,\partial \Omega\\
u_2\;\,\text{is}\;\,z\text{-outgoing}
\end{cases}
\end{equation}
Here, $\partial_\nu$ and $\partial_{\nu'}$ are respectively the interior and exterior normal derivatives of $u$ at $\partial\Omega\,.$

\subsection{Reduction to Transcendental Equations on the Circle}

We now consider \eqref{eqn:main} with $\Omega=B(0,1)\subset \re^2$ and $V\equiv h^{-\alpha}  V_0$ on $\partial \Omega$. Then for $i=1,2$,
\begin{equation}
\label{eqn:circleMain}
\begin{cases} \left(-h^2\partial_r^2-\frac{h^2}{r}\partial_r-\frac{h^2}{r^2}\partial_{\theta}^2-z^2\right)u_{1}=0&\text{in } B(0,1)\\
\left(-h^2\partial_r^2-\frac{h^2}{r}\partial_r-\frac{h^2}{r^2}\partial_{\theta}^2-z^2\right)u_{2}=0&\text{in } \re^2\setminus B(0,1)\\
u_1(1,\theta)=u_2(1,\theta)\\
\partial_ru_1(1,\theta)-\partial_ru_2(1,\theta)+Vu_1(1,\theta)=0\\
u_2\text{ is }z \text{ outgoing}
\end{cases}.
\end{equation}
Expanding in Fourier series, write $u_{i}(r,\theta):=\sum_{n}u_{i,n}(r)e^{in\theta}.$
Then, $u_{i,n}$ solves
$$\left(-h^2\partial_r^2-h^2\recip{r}\partial_r+h^2\frac{n^2}{r^2}-z^2\right)u_{i,n}(r)=0.$$
Multiplying by $r^2$ and rescaling by $x=zh^{-1}r$, we see that $u_{i,n}(r)$ solves the Bessel equation with parameter $n$ in the $x$ variables. Then, using that $u_2$ is outgoing and $u_1\in L^2$, we obtain that $$u_{1,n}(r)=K_nJ_n(zh^{-1} r)\quad\text{ and }\quad  u_{2,n}(r)=C_nH_n^{(1)}(zh^{-1} r)$$
where $J_n$ is the $n^{\text{th}}$ Bessel function of the first kind, and $H_n^{(1)}$ is the $n^{\text{th}}$ Hankel function of the first kind.

To solve \eqref{eqn:circleMain} and hence find a resonance, we only need to find $z$ such that the boundary conditions hold. 
Using the boundary condition $u_1(1,\theta)=u_2(1,\theta)$, we have $K_nJ_n(zh^{-1})=C_nH_n^{(1)}(zh^{-1})$. Hence, 
$$C_n=\frac{K_nJ_n(zh^{-1})}{H_n^{(1)}(zh^{-1})}.$$ 
Next, we rewrite the second boundary condition in \eqref{eqn:circleMain} and use that $V\equiv h^{-\alpha}V_0$ to get
$$\sum_n(K_nzh^{-1} J_n'(zh^{-1})-C_nzh^{-1} H_n^{(1)'} (zh^{-1})+h^{-\alpha}V_0K_n J_n(zh^{-1}))e^{in\theta}=0.$$
Then, since $e^{in\theta}$ are $L^2$ orthogonal, we have 
$$K_n\left(zh^{-1} J_n'(zh^{-1})-zh^{-1}\frac{J_n(zh^{-1})}{H_n^{(1)}(zh^{-1})}H_n^{(1)'}(zh^{-1})+h^{-\alpha}V_0J_n(zh^{-1})\right)=0\,,\quad n\in \mathbb{Z}.$$
Thus
\begin{equation*}
\label{eqn:preWronskian}
K_nh^{-\alpha}V_0= K_nzh^{-1}\left(\frac{H_n^{(1)'}(zh^{-1})}{H_n^{(1)}(zh^{-1})}-\frac{J_n'(zh^{-1})}{J_n(zh^{-1})}\right).
\end{equation*}
which can be written
\begin{equation}
\label{eqn:diskResonances}
h^{-\alpha}V_0K_n=K_nzh^{-1}\frac{W(J_n,H_n^{(1)})}{J_n(zh^{-1})H_n^{(1)}(zh^{-1})}=\frac{2 iK_n}{\pi J_n(zh^{-1})H_n^{(1)}(zh^{-1})}
\end{equation}
where $W(f,g)$ is the Wronskian of $f$ and $g$.

Then, without loss, we assume $K_n=1$ or $K_n=0$. 
Hence, we seek solutions $z(h,n)$ to 
\begin{equation} 
\label{eqn:toSolveExact}1-\frac{\pi h^{-\alpha}V_0}{2i}J_n(h^{-1}z(h,n))H_n^{(1)}(h^{-1}z(h,n))=0.
\end{equation}

The quantity $nh^{-1}$ is the tangential frequency of the mode $u_{i,n}e^{in\theta}.$ In particular, the \emph{wave front set,} denoted $\WFh$ (see \cite[Chapter 4]{EZB}), of $e^{in\theta}$ has 
$$\WFh(e^{in\theta})\subset \{\xi'=nh\mod \o{}(1)\}.$$
 Thus, $|n|<(1-\e)h^{-1}$ corresponds to modes concentrating near directions transverse to the boundary,  $|n|\sim h^{-1}$ are the glancing frequencies, that is directions tangent to the boundary, and $|n|>(1+\e)h^{-1}$ corresponds to elliptic frequencies.

\section{Resonance Free Regions}
\label{sec:resFreeCircle}
In this section, we demonstrate the existence of resonance free regions. In particular, we prove Theorem \ref{thm:resFreeCircle}. We write $n=mh^{-1}$ and assume that 
$$|\Im z|\leq M_0\min(h\log h^{-1}, h^{2\alpha-2/3}).$$
\subsection{Analysis away from glancing ($m\leq 1-\e$ and $m\geq 1+\e$)}
Away from glancing, that is for $|h|n|-1|\geq ch^\delta$ with $0\leq \delta <1/2$, we apply \cite[Lemma 6.1]{Galk}. First, we compute the chord length of a trajectory starting in $\partial B(0,1)$ with a given slope. Let $\gamma(t)$ be a line with slope $r$ through $(0,-1)$. Then we find it's second intersection with $\partial B(0,1)$. 
$$\gamma(t):=(t,rt-1)$$
Then, $|\gamma(t)|=1$ implies 
$$(r^2+1)t^2-2rt=0.\quad \text{ Hence}\quad t=\frac{2r}{r^2+1}=2\frac{r}{\sqrt{1+r^2}}\frac{1}{\sqrt{1+r^2}}=\frac{2((1,r)\cdot (0,1))}{|(1,r)|^2}.$$

The unit tangent vector to $\gamma$ is given by $(1,r)/\sqrt{1+r^2}.$ Hence, the chord length is given by $2\xi_\nu$ where $\xi_\nu$ is the normal component of the unit tangent vector to $\gamma$. 

Using this in \cite[Lemma 6.1]{Galk}, we have that for $0\leq \delta <1/2$, there are no resonances contributed from the region $|m-1|\geq ch^\delta$ when
\begin{equation}\label{eqn:awayGlance}-\frac{\Im z}{h}\leq \inf_{|\xi'|<1-ch^{\delta}}\frac{1}{4\sqrt{1-|\xi'|^2}}\log \left(\frac{ h^{2-2\gamma}V_0}{4(1-|\xi'|^2)+h^{2-2\gamma}V_0}\right).\end{equation}

\subsection{Analysis for $1-\e h^\delta \\leq m\leq 1+\e h^\delta $}
\label{sec:modeldeltaD}
In this section, we consider the remaining values of $m$. First, we use \eqref{eqn:deltaMainAsymptotics} in \eqref{eqn:toSolveExact} to write
\begin{equation}
\label{eqn:AiryExact}
1-2\pi h^{-\alpha}V_0e^{-5\pi i/6}\left(\frac{\zeta}{1-(m^{-1}z)^2}\right)^{1/2}\left(\frac{Ai(n^{2/3}\zeta)A_-(n^{2/3}\zeta)}{n^{2/3}}+\O{}(Ei_-(\tfrac{8}{3},2,\tfrac{10}{3}))\right)=0
\end{equation}
where $\zeta=\zeta(m^{-1}z)$. 
We first ignore the error term in \eqref{eqn:AiryExact} and show that there are no solutions with the appropriate bounds on $\Im \zeta$. In particular, define $h_1:=n^{-1}$ so that $(1-\e h^{\delta})h\leq h_1\leq (1+\e h^{\delta})h$ and 
$$\W:=h_1^{2/3}h^{-\alpha }\left(\frac{\zeta}{1-( m^{-1}z)^2}\right)^{1/2}V_0=\O{}(h^{2/3-\alpha}).$$
The fact that $\W$ has uniform bounds for $\zeta$ in the relevant region comes from the fact that $h_1h^{-1}=m$ and $1-\e h^{\delta}<m<1+\e h^{\delta}$. 
Then, rewriting \eqref{eqn:AiryExact} without the lower order terms, we have
\begin{equation}
\label{eqn:AiryToSolve}
1-2\pi e^{-5\pi i/6}\W(\zeta)A_-(h_1^{-2/3}\zeta)Ai(h_1^{-2/3}\zeta)=0.
\end{equation}
Notice that if $\alpha \geq 2/3$ and $Mh^{2-2\alpha}\leq |\Re \zeta|\leq Ch^\delta$ or $\alpha<2/3$ and $|\Re \zeta|\leq Ch^\delta$ for any $\delta>0$, then the second term in \eqref{eqn:AiryExact} is bounded above by $1-\delta_1$ for some $\delta_1>0$. Hence, \eqref{eqn:AiryExact} has no solutions and we need only consider the remaining $\Re \zeta$.

\subsection{Analysis at glancing ($m\sim 1$)}
We next analyze $|\zeta|<M\max(h^{2/5(3-2\alpha)},h^{2/3}).$ Let $s=h^{-2/3}\Re \zeta$. then, $$0\leq |s|<M\max(h^{2/5(3-2\alpha)-2/3},1) $$ and 
$$\zeta=h^{2/3}s+\Im \zeta=h^{2/3}s+\O{}(\min(h\log h^{-1}, h^{2\alpha-2/3})).$$
Thus, 
\begin{multline*} 
|\W(\zeta)AiA_-(h_1^{-2/3}\zeta)-\W(h_1^{2/3}s)AiA_-(s) -\W(h_1^{2/3}s)(AiA_-)'(s)i\Im h_1^{-2/3}\zeta|\leq\\\O{}(h^{2/3}\la s\ra^{1/2}h^{-\alpha}(\Im h^{-2/3}\zeta)^2)+\O{}(h^{2/3}h^{-\alpha}\Im \zeta AiA_-(h^{-2/3}\zeta) )
\end{multline*}
We obtain lower bounds on 
$$f(s,h,h_1):=1-2\pi e^{-5\pi i/6}\W(h_1^{2/3}s)\left(A_-Ai(s)+(A_-Ai)'(s)i\Im h_1^{-2/3}\zeta\right).$$

Letting $\alpha:=e^{-5\pi i/6}$, we have by \eqref{eqn:ImAMinus} that
$$\alpha A_-(s)Ai(s)= \Re (\alpha A_-(s))Ai(s)-i\frac{Ai^2(s)}{2}$$
and 
\begin{align*} 
(\alpha A_-Ai)'(s)i\Im h_1^{-2/3}\zeta&= \left(Ai(s)Ai'(s) +\right.,\\
&\quad\left.i\left[Ai'(s)\Re (\alpha A_-(s)) +Ai(s)\Re (\alpha A_-'(s))\right]\right) \Im h_1^{-2/3} \zeta
\end{align*}

Thus, 
$$\Im f=- 2\pi \W(h_1^{2/3}s)\left( -\frac{Ai^2(s)}{2}+\left(Ai'(s)\Re(\alpha A_-(s))+Ai(s)\Re(\alpha A_-'(s))\right)\Im h_1^{-2/3}\zeta\right) $$
and 
$$\Re f=1-2\pi \W(h_1^{2/3}s)Ai(s)\left(\Re (\alpha A_-(s)) +Ai'(s)\Im h_1^{-2/3}\zeta\right).$$
So, when
$$|Ai(s)|\leq \frac{1-\delta}{2\pi \W(h_1^{2/3}s)\Re(\alpha A_-(s))},\quad \text{ or }\quad |Ai(s)|\geq \frac{1-\delta}{2\pi \W(h_1^{2/3}s)\Re(\alpha A_-(s))}$$
then $|f|\geq \delta$. Note that for $\alpha<2/3$, this condition is never satisfied.
Thus, we need only consider 
\begin{equation}\label{eqn:range} \frac{1-\delta}{2\pi \W(h_1^{2/3}s)\Re(\alpha A_-(s))}\leq |Ai(s)|\leq\frac{1-\delta}{2\pi \W(h_1^{2/3}s)\Re(\alpha A_-(s))}.\end{equation}
That is, using the fact that $|Ai'(-s)|\sim c|s|^{1/4}$ and $|A_-(-s)|\sim c|s|^{-1/4}$,
\begin{equation} \label{eqn:sLoc} s=-\zeta_k+\O{}( h^\alpha h^{-2/3}).
\end{equation}
where $-\zeta_k$ is the $k^{\text{th}}$ zero $Ai(s)$.

\begin{remark}
For $\alpha\leq 2/3$, notice that \eqref{eqn:sLoc} does not give us any additional information on the location of $s$. However, it is easy to see that in this situation $\Im f\geq Ch^{\alpha-2/3}.$ Since we need only consider small $\Re \zeta$ when $\alpha\geq 2/3$, this implies that in the relevant region $|\Im f|\geq c$ and hence there are no solutions to \eqref{eqn:toSolveExact} in this region.
\end{remark}

Now, $|\Im z|\leq M_0\min(h\log h^{-1}, h^{2\alpha-2/3})$ implies that $|\Im \zeta|\leq M_1h^{2\alpha-2/3}$. So, using the fact that $A_-(-s)=\O{}(|s|^{-1/4})$ and $Ai'(-s)=\O{}(|s|^{1/4})$ we see that there exists $K=K(M_1)$ such that if for some $\delta_1>0$ small enough,
\begin{equation} 
\label{eqn:resFreeSize} 
\inf_{k\leq K(M_1)}\left|\Im \zeta-\frac{ h_1^{2/3}}{8\pi ^2 \W(h_1^{2/3}(-\zeta_k))^2\Re(\alpha A_-(-\zeta_k))^3Ai'(-\zeta_k)}\right|\geq \delta_1 h^{2\alpha-2/3}
\end{equation}
then 
$$|\Im f|\geq c\delta_1 h^{\alpha-2/3}.$$

Finally, we account for the error terms.  We have suppressed terms of the form
$$\O{}(\max(1,h^{1/5(3-2\alpha)-1/3})\min(h^{4\alpha-8/3}, h^{2/3}(\log h^{-1})^2)+h^{\alpha-2/3}AiA_-(h^{-2/3}\zeta)).$$
Together with \eqref{eqn:range}
the estimate $|f|\geq c\delta_1 h^{\alpha+2/3}$ implies that there are no solutions to \eqref{eqn:toSolveExact} for $|\Re \zeta|<M\max(h^{2/5(3-2\alpha)},h^{2/3})$, $|\Im \zeta|<M_1\min(h^{2\alpha-2/3},h\log h^{-1})$, satisfying \eqref{eqn:resFreeSize}.

\subsection{Asymptotic analysis near glancing}

We need to analyze $C\e h^{\delta}\geq |\Re \zeta|\geq h^{2/5(3-2\alpha)}M.$ To do this, let $G_\Delta=\frac{ih_1^{1/3}}{2(-\zeta)^{1/2}}$, and $b:=2\pi e^{-5\pi i/6}.$
Then if $\zeta$ solves \eqref{eqn:AiryExact}
$$G_{\Delta}^{1/2}\W=-G_{\Delta}^{1/2}\W G_{\Delta}^{1/2}(-G_{\Delta}^{-1/2}bAi(h_1^{-2/3}\zeta)A_-(h_1^{-2/3}\zeta)G_{\Delta}^{-1/2}+\O{}(h^2))G_{\Delta}^{1/2}\W$$
and hence
\begin{multline*} (1+G_{\Delta}^{1/2}\W G_{\Delta}^{1/2})G_{\Delta}^{1/2}\W\\
=-G_{\Delta}^{1/2}\W G_{\Delta}^{1/2}(-G_{\Delta}^{-1/2}bAi(h_1^{-2/3}\zeta)A_-(h_1^{-2/3}\zeta)G_{\Delta}^{-1/2}-1+\O{}(h^2))G_{\Delta}^{1/2}\W.\end{multline*}

Using \eqref{eqn:AiryAsympPos} for $\Re \zeta<-Mh^{2/3}$, we have 
$$(1+G_{\Delta}^{1/2}\W G_{\Delta}^{1/2})G_{\Delta}^{1/2}\W=-G_{\Delta}^{1/2}\W G_{\Delta}^{1/2}(-ie^{\frac{4i}{3h_1}(-\zeta)^{3/2}}(1+\O{}(h\zeta^{-3/2})))G_{\Delta}^{1/2}\W.$$
$$G_\Delta^{1/2}\W=-(I+G_{\Delta}^{1/2}\W G_{\Delta}^{1/2})^{-1}G_{\Delta}^{1/2}\W G_{\Delta}^{1/2}(-ie^{\frac{4i}{3h_1}(-\zeta)^{3/2}}(1+\O{}(h\zeta^{-3/2})))G_{\Delta}^{1/2}\W.$$
For $\zeta>Mh^{2/3}$, we use \eqref{eqn:AiryAsympPos} to obtain
$$(I+G_{\Delta}^{1/2}\W G_{\Delta}^{1/2})G_{\Delta}^{1/2}\W=-G_{\Delta}^{1/2}\W G_{\Delta}^{1/2}(\O{}(h\zeta^{-3/2}))G_{\Delta}^{1/2}\W.$$
Hence,
$$G_{\Delta}^{1/2}\W=-(I+G_{\Delta}^{1/2}\W G_{\Delta}^{1/2})^{-1}G_{\Delta}^{1/2}\W G_{\Delta}^{1/2}\O{}(h\zeta^{-3/2})G_{\Delta}^{1/2}\W.$$

To see that $I+G_{\Delta}^{1/2}\W G_{\Delta}^{1/2}\neq 0$ observe that when $\Re \zeta<-Mh^{2/3}$, 
$$\left|\Re \frac{ih_1^{1/3} \W}{2(-\zeta)^{1/2}}\right|=h_1^{1-\alpha}\O{}\left(\frac{\Im \zeta}{|\Re \zeta|^{3/2}}\right)=\o{}(1)$$
and when $\Re \zeta>Mh^{2/3}$, 
$$\Re \frac{h_1^{1/3}\W}{2\zeta^{1/2}}\geq 0.$$

Now, since $|\Re \zeta|>Mh^{2/3}$, $\O{}(h\zeta^{-3/2})\ll 1$ for $M$ large. Hence, there are no zeros for $\Re \zeta >0$. For $\Re \zeta <0$, there are no zeros of \eqref{eqn:AiryToSolve} when
\begin{equation} 
\label{eqn:AbsVal}\left|(I+G_{\Delta}^{1/2}\W G_{\Delta}^{1/2})^{-1}G_{\Delta}^{1/2}\W G_{\Delta}^{1/2}(1+\O{}(h\zeta^{-3/2}))e^{\frac{4i}{3h_1}(-\zeta)^{3/2}}\right|<1.
\end{equation}

Let $\zeta=s+i\Im \zeta$. Then
$$(-\zeta)^{3/2}=(-s)^{3/2}(1-i\Im \zeta (-s)^{-1})^{3/2}=(-s)^{3/2}\left(1-\frac{3}{2}i\Im \zeta (-s)^{-1}+\O{}((\Im \zeta)^2s^{-2})\right)$$
and
$$(-\zeta)^{1/2}=(-s)^{1/2}(1+\O{}(\Im \zeta s^{-1})).$$
So, 
$$|e^{\frac{4i}{3h_1}(-\zeta)^{3/2}}|=e^{\frac{2\Im \zeta(-s)^{1/2}}{h_1 }+\O{}((\Im \zeta)^2|s|^{-1/2}h^{-1})}.$$
Taking logarithms of \eqref{eqn:AbsVal},
$$\frac{2\Im \zeta(-s)^{1/2}}{h_1}+\O{}((\Im \zeta)^2h^{-1}|s|^{-1/2})+\log\left|\frac{ h_1^{1/3}\W}{2i(-\zeta)^{1/2}-h_1^{1/3}\W}\right|+\O{}(h\zeta^{-3/2})<0.$$

Thus, for $-C\e h^{\delta}\leq\Re \zeta=s\leq -M\max(h^{2/5(3-2\alpha)}, h^{2/3})$, there are no solutions with
\begin{multline*} \Im \zeta<\inf_{-C\e h^{\delta}<s<-M\max(h^{2/5(3-2\alpha)},h^{2/3})}\frac{h_1}{4(-s)^{1/2}}\log \left|1+4(-s)h_1^{-2/3}\W^{-2}\right|\\+\O{}((\Im \zeta)^2|s|^{-1}+\Im \zeta |s|^{-1}h)+\O{}(h^2|s|^{-2}).\end{multline*}

\begin{tabular}{|c|c|c|}
\hline
&Main Term&Error\\
\hline
 $|s|< h^{2-2\alpha}$&$ h_1^{1/3}h^{2\alpha-4/3}(-s)^{1/2}$ &$|s|^{-1}\left(\begin{gathered} h^2|s|^{-1}+\\\min(h^2(\log h^{-1})^2,hh^{2\alpha-2/3}\end{gathered}\right)$\\
 \hline
$|s|\geq h^{2-2\alpha} $&$ \frac{h_1}{(-s)^{1/2}}\log (1-sh^{2\alpha-4/3}h_1^{-2/3})$&$h^2s^{-2}+\O{}(h^{4/3}(\log h^{-1})^2)$\\
\hline
\end{tabular}

Thus, since we have $|s|>M\max(h^{2/5(3-2\alpha)},h^{2/3})$, the error terms are lower order and hence 
$$ \Im \zeta<\inf_{-C\e h^{\delta}<s<-Mh^{2/5(3-2\alpha)}}\frac{h_1}{4(-s)^{1/2}}\log \left|1+4(-s)h_1^{-2/3}\W^{-2}\right|.$$

So, for $1-\e h^{\delta}\leq m\leq 1+\e h^{\delta}$, and $C\e h^{\delta}>|s|>Mh^{2/5(3-2\alpha)}$ there are no zeros of \eqref{eqn:toSolveExact} for 
\begin{equation}\label{eqn:resFreehyp} \Im \zeta\leq C\min(M^{1/2}h^{2\alpha-2/3}, Ch^{1-\delta/2}\log h^{-1}).\end{equation}
Taking $M$ large enough depending on $K$ and $h$ small enough, 
$$CM^{1/2}h^{2\alpha-2/3}\gg \sup_{k\leq K}\frac{ h_1^{2/3}}{8\pi ^2 \W(h_1^{2/3}(-\zeta_k))^2\Re(\alpha A_-(-\zeta_k))^3Ai'(-\zeta_k)}.$$
Moreover,
$$h^{1-\delta/2}\log h^{-1}\gg \frac{1-\alpha}{2}h\log h^{-1}.$$ Thus, the region $C\e h^{\delta}>|s|>M\max(h^{2/5(3-2\alpha)},h^{2/3})$ does not contribute to the resonances closest to the real axis.

Our last task is to relate the imaginary part of $z$ to that of $\zeta$ when $|\zeta|<\e h^{\delta}$. By \eqref{eqn:diffEqzeta} and \eqref{eqn:zetaLimit}, we have that 
\begin{equation}
\label{eqn:zZeta}z=h_1^{-1}h-h_1^{-1}h\frac{\zeta}{\sqrt[3]{2}}+\O{}(\zeta^2)\,,\quad \Im z=-h_1^{-1}h\frac{\Im \zeta}{\sqrt[3]{2}}+\O{}(\Re \zeta \Im \zeta).
\end{equation}
More generally $\Im z\sim C\Im \zeta+\O{}(|\Im \zeta|^2)$ for $|\zeta|<M$.
Since we assume $\Re z\in[1-Ch^{3/4},1+Ch^{3/4}]$ and we have $h_1=h+\O{}(h^{1+\delta})$,  \eqref{eqn:awayGlance}, \eqref{eqn:resFreeSize}, \eqref{eqn:resFreehyp}, \eqref{eqn:zZeta} 
and the fact that 
$$\lim_{w\to 1}\Phi^2(w)=(\sqrt[3]{2})^{-2}$$
complete the proof of the existence of resonance free regions of the sizes given in Theorem \ref{thm:resFreeCircle}.

\section{Construction of Resonances}
\label{sec:existence}
In this section, we demonstrate the existence of resonances. That is, we prove Theorem \ref{thm:existenceCircle}.  We first prove the following analog of Newton's method:
\begin{lemma}
\label{lem:Newton}
Suppose that $z_0\in\mathbb{C}$. Let 
$\Omega:=\{z\in\complex: |z-z_0|\leq \e\}$
and suppose $f:\Omega\to \mathbb{C}$ is analytic. Suppose that 
$$|f(z_0)|\leq a\,,\quad |\partial_zf(z_0)|\geq b\,,\quad \sup_{z\in \Omega}|\partial_z^2f(z)|\leq d.$$
Then if 
\begin{equation}
\label{eqn:condition}
a+d\e^2<\e b<c<1
\end{equation}
there is a unique solution $z$ to $f(z)=0$ in $\Omega$.
\end{lemma}
\begin{proof}
Let 
$$g(z):=z-\frac{f(z)}{\partial_zf(z_0)}.$$
Then, 
$$|\partial_zg(z)|=\left|1-\frac{\partial_zf(z)}{\partial_zf(z_0)}\right|\leq \frac{d\e}{b}$$
and 
$$|g(z)-z_0|\leq |g(z_0)-z_0|+\sup_{\Omega}|\partial_zg(z)||z-z_0|\leq \frac{a}{b}+\frac{d\e^2}{b}.$$
Thus under the condition \eqref{eqn:condition}, $g:\Omega\to \Omega$ and 
$$|g(z)-g(z')|<\sup_{w\in\Omega}|\partial_zg(w)||z-z'|<c|z-z'|.$$
Hence, $g$ is a contraction mapping and by the contraction mapping theorem, there is a unique fixed point of $g$ in $\Omega$ and hence a zero of $f(z)$ in $\Omega$. 
\end{proof}

\subsection{Resonances at glancing}
We now analyze $n\sim h^{-1}$ which correspond to glancing trajectories. In particular, for $\alpha>2/3$, we construct solutions to \eqref{eqn:toSolveExact} for $0<h<h_0$ with 
$$\Im z\geq Ch^{2\alpha -2/3}.$$
Let $h_1=n^{-1}$. Then, suppressing terms of size $h_1^{2+2/3}h^{-\alpha}$, we seek solutions to \eqref{eqn:AiryToSolve}.
Our ansatz is 
$$\zeta= -h_1^{2/3}\zeta_k+\e(h)$$
where $-\zeta_k$ is the $k^{\text{th}}$ zero of $Ai(s).$ Then, 
\begin{multline*} \W(\zeta)A_-(h_1^{-2/3}\zeta)Ai(h_1^{-2/3}\zeta)=
\left(\W(-h_1^{2/3}\zeta_k)+\sum_{k\geq 1} \frac{\W^{(k)}(-h_1^{2/3}\zeta_k)}{k!}\e^k\right)\\\left(A_-Ai'(-\zeta_k)h_1^{-2/3}\e +A_-'Ai'(-\zeta_k)h_1^{-4/3}\e^2+\sum_{k\geq 3}\frac{(A_-Ai)^{(k)}(-\zeta_k)}{k!} h_1^{-2k/3}\e^k\right).
\end{multline*}

Let $\e=\e_0+\e_1$ where $\e_1=\o{}(\e_0)$. Then, ignoring terms terms of size $\e^2$ and letting $b:=2\pi e^{-5\pi i/6}$, we have
$$1-b\W(-h_1^{2/3}\zeta_k)A_-Ai'(-\zeta_k)h_1^{-2/3}\e_0=0.$$
That is,
$$ \e_0=\frac{h_1^{2/3}}{b\W(-h_1^{2/3}\zeta_k)A_-(-\zeta_k)Ai'(-\zeta_k)}=Ch_1^{2/3}h^{\alpha}h_1^{-2/3}$$
Then, using terms of size $\e_0^2$ and $\e_1$, we have 
\begin{multline*} \W(-h_1^{2/3}\zeta_k)A_-Ai'(-\zeta_k)h_1^{-2/3}\e_1\\+(\W(-h_1^{2/3}\zeta_k)A_-'Ai'(-\zeta_k)h_1^{-2/3}
+\W'(-h_1^{2/3}\zeta_k)A_-A'(-\zeta_k))h_1^{-2/3}\e_0^2=0.
\end{multline*}
That is, 
\begin{align*} 
\e_1&=-\frac{h_1^{2/3}(\W(-h_1^{2/3}\zeta_k)A_-'Ai'(-\zeta_k)h_1^{-4/3}+\W'(-h_1^{2/3}\zeta_k)A_-Ai'(-\zeta_k)h_1^{-2/3})\e_0^2}{\W(-h_1^{2/3}\zeta_k)A_-Ai'(
-\zeta_k)}\\
&=-h_1^{-2/3}\e_0^2\frac{A_-'(-\zeta_k)}{A_-(-\zeta_k)}(1+\O{}(\e_0h_1^{-2/3}))\\
&=-\frac{h_1^{2/3}A_-'(-\zeta_k)}{(\W(-h_1^{2/3}\zeta_k))^24\pi ^2e^{-5\pi i/3}A_-^3(-\zeta_k)(Ai'(-\zeta_k))^2}\left(1+\O{}(\e_0h_1^{-2/3})\right).
\end{align*}
So, since by \eqref{eqn:ImAMinus} 
$$\Im (e^{-5\pi i/6}A_-(s))=-\frac{Ai(s)}{2}.$$
\begin{align*}
\Im \e_1&=-\frac{h_1^{2/3}\Im(e^{-5\pi i/6}A_-'(-\zeta_k))}{(\W(-h_1^{2/3}\zeta_k)))^24\pi ^2(e^{-5\pi i/6})^3A_-^3(-\zeta_k)(Ai'(-\zeta_k))^2}\left(1+\O{}(\e_0h_1^{-2/3})\right)\\
&=\frac{h_1^{2/3}}{(\W(-h_1^{2/3}\zeta_k))^28\pi ^2(e^{-5\pi i/6})^3A_-^3(-\zeta_k)Ai'(-\zeta_k)}\left(1+\O{}(\e_0h_1^{-2/3})\right)\\
\end{align*}
 Since $b \Phi(-h_1^{2/3}\zeta_k)A_-Ai'(-\zeta_k)\neq 0$, repeating in this way we obtain an asymptotic expansion for $\e(h)$ in powers of $h^{\alpha}h_1^{-2/3}$ such that for $\zeta=-h_1^{2/3}\zeta_k+\e(h)$, 
$$1-b\W(\zeta)A_-(h_1^{-2/3}\zeta)Ai(h_1^{-2/3}\zeta)=\O{}(h_1^\infty).$$

Let 
$$f(\zeta)=1-b\W(\zeta)A_-(h_1^{-2/3}\zeta)Ai(h_1^{-2/3}\zeta).$$
Then, for $\zeta=-h_1^{2/3}\zeta_k+\O{}(h_1^{\alpha})$, 
$$|f'(\zeta)|\geq ch^{-\alpha}$$
and
$$|f''(\zeta)|\leq Ch^{-\alpha}h_1^{-2/3}.$$
Thus, letting $n=h^{-1}+\O{}(1)$ and using Lemma \ref{lem:Newton}, there is a solution $\zeta_0(h_1,h)$ to $f(\zeta_0(h_1,h))=0$ with $$\zeta_0=-h_1^{2/3}\zeta_k+\e(h)+\O{}(h^\infty).$$

Now, by the implicit function theorem  (or Rouche's theorem)
$f(\zeta)=a(\zeta)$ defines $\zeta$ in a neighborhood of $\zeta_0$ for $a$ small enough. Hence, since we suppressed terms of size $h_1^{8/3-\alpha}$ in \eqref{eqn:deltaMainAsymptotics}, we have that there is a resonance with 
$$\zeta=\zeta_0+\frac{\O{}(h_1^{8/3}h^{-\alpha})}{\partial_\zeta f(\zeta_0)}=\zeta_0+\O{}(h_1^{8/3}).$$

\subsection{Resonances normal to the boundary}
Next, we consider $n$ fixed relative to $h$. That is, we consider modes that concentrate normal to $\partial B(0,1)$. 

Using asymptotics \eqref{eqn:normalAsymptotics} in \eqref{eqn:toSolveExact}, we have
\begin{equation}
\label{eqn:asympToSolve}
1-\frac{h^{1-\alpha}V_0}{2i z(h,n)}\left(e^{2iz(h,n)/h-(n+\tfrac{1}{2})\pi i}(1+\O{}(hz(h,n)^{-1}))+1\right)=0.
\end{equation}
Let 
$$F(\e,k,n,h)=1-\frac{2h^{1-\alpha}V_0}{i\pi h(4k+2n+1)}\left(e^{2i\e/h}+1\right).$$

Then, 
$$\e_0(k,n,h)=\frac{-ih}{2}\log\left[h^{\alpha-1}\frac{i\pi h(4k+2n+1)}{2V_0}-1\right]$$
has 
$$F(\e_0(k,n,h),k,n,h)=0,\quad |\partial_\e F(\e_0(k,n,h),k,n,h)|\geq ch^{-1}.$$

Now, for $0<c$ and $ch^{-1}<k<Ch^{-1}$ by \eqref{eqn:asympToSolve}, $z(h,k,n)$ can be defined by a solution
$z(h,k,n)=\frac{\pi h}{4}(4k+2n+1)+\e(k,n,h)$ where 
$$F(\e,k,n,h)=\O{}(e^{2i\e/h}h/z+\e).$$
So, by the implicit function theorem there is a solution $\e$ satisfying
\begin{align*} 
\e(k,n,h)&=\e_0(k,n,h)+(\partial_\e F(\e_0(k,n,h),k,n,h))^{-1}\O{}(h^{1-\alpha}e^{2i\e_0/h}(h/z+\e_0))\\
&=\e_0(k,n,h)+\O{}(h^{2}).
\end{align*}

Thus, for all $\e>0$ and $0<h<h_\e$, there exist $z(h)\in \Lambda$ with
\begin{equation}\label{eqn:normal}
\frac{\Im z}{h}\sim\begin{cases} -\frac{(1-\alpha)}{2}\log h^{-1}+\frac{1}{2}\log\left(\frac{2}{V_0}\right)+\O{}(h^{3/4})&\alpha<1\\
-\frac{1}{4}\log \left(1+\frac{4}{V_0^2}\right)+\O{}(h^{3/4})&\alpha=1\end{cases}
\end{equation}
\begin{remark} Note that the size of the error terms in \eqref{eqn:normal} comes from the fact that we allow $\Re z \in [1-Ch^{3/4},1+Ch^{3/4}]$. 
\end{remark}

\noindent This completes the proof of Theorem \ref{thm:existenceCircle}.

\subsection{Resonances Away from Glancing}
Finally, we construct resonances coming from modes concentrating farther away from glancing but not normal to the boundary. In particular, we show the existence of modes concentrating $h^{2/3-2\e/3}$ of glancing for $(3\alpha-2)/4<\e \leq 1.$ This will prove Theorem \ref{thm:lowerBound}.

To do this, let $w=(nh)^{-1}z$ and $\zeta=\zeta(w)$. Then we first suppress the lower order terms in \eqref{eqn:deltaMainAsymptotics} and solve \eqref{eqn:AiryToSolve}. Using the asymptotics \eqref{eqn:AiryAsympPos}, in \eqref{eqn:AiryToSolve} and letting $n=h_1^{-1}$ we have 
\begin{multline}\label{eqn:toMultiply} 1- \frac{h_1^{1/3}\W i}{2(-\zeta)^{1/2}}\left(1+\sum_{j=1}^{N-1}\frac{c_kh_1^{k}}{(-\zeta)^{3k/2}}-ie^{\frac{4}{3h_1}i(-\zeta)^{3/2}}\left(1+\sum_{k=1}^{N-1}\frac{b_kh_1^{k}}{(-\zeta)^{3k/2}}\right)\right)\\
+\O{}\left(h_1^{N+1}h^{-\alpha}(-\zeta)^{-(3N+1)/2}(1+e^{\frac{4}{3h_1}i(-\zeta)^{3/2}})\right)=0
\end{multline}
where $c_k$ and $b_k$ are real.

We make the ansatz
\begin{equation}\label{eqn:ansatz} 
(-\zeta)^{3/2}=\frac{3}{8}\pi h_1(4k-1)+\e=:m+\e
\end{equation}
where we assume $\e=\O{}(mh^{\delta})$ for some $\delta>0$.
Then,
\begin{equation} \label{eqn:zoneHalf}(-\zeta)^{1/2}=m^{1/3}\left(1+\frac{1}{3m}\e+\O{}(\e^2/m^2)\right)\,, \quad (-\zeta)=m^{2/3}\left(1+\frac{2}{3m}\e +\O{}(\e^2/m^2)\right).\end{equation}
and $ie^{\frac{4}{3h_1}i(-\zeta)^{3/2}}=e^{\frac{4}{3h_1}\e}.$
Multiplying \eqref{eqn:toMultiply} by $(-\zeta)^{1/2}$ and using 
$$\W(\zeta)=\sum_{n=0}^{N-1}\frac{\W^{(n)}(m)\e^n}{n!}+\O{}(h_1^{2/3}h^{-\alpha}\e^N),$$
we have
\begin{multline}
\label{eqn:approxSolve}
(-\zeta)^{1/2}-\frac{h_1^{1/3}\W(m)i}{2}\left(1+\sum_{k=1}^{N-1}\frac{c_kh_1^{k}}{(-\zeta)^{3k/2}}-ie^{\frac{4}{3h_1}i(-\zeta)^{3/2}}\left(1+\sum_{k=1}^{N-1}\frac{b_kh_1^{k}}{(-\zeta)^{3k/2}}\right)\right)\\
+\O{}\left(h_1^{1}h^{-\alpha}(h_1^{N}m^{-N}+\e)(1+e^{\frac{4}{3h_1}i\e}\right)=0.\end{multline}

Then, let $\e(h)=\e_0+\e_1$ where $\e_1=\O{}(\e_0h^{\delta})$ for some $\delta>0$. Using terms which do not involve $\e$ and the exponential term,
$$\e_0=-\frac{3h_1i}{4}\left[\log \left(\frac{2m^{1/3}i}{h_1^{1/3}\W(m)}+1+\sum_ {k=1}^{N-1}c_kh^{k}m^{-k}\right)-\log\left(1+\sum_{k=1}^{N-1}b_kh^km^{-k}\right)\right].$$
Now, using 
$$e^{\frac{4i}{3h_1}(\e_0+\e_1)}=e^{\frac{4i}{3h_1}\e_0}\left(1+\frac{4i}{3h_1}\e_1+\O{}(\e_1^2h^{-2})\right).$$
we can solve for an asymptotic expansion for $\e(h)$ in powers of $h_1m^{-1}$ so that for $(-\zeta_0)^{3/2}=m+\e(h)$,
$$(-\zeta_0)^{1/2}- \frac{h^{1/3}\W(\zeta_0)i}{2}\left(1+\sum_{j=1}^{N-1}\frac{c_kh^{k}}{(-\zeta_0)^{3k/2}}-ie^{\frac{4}{3h}i(-\zeta_0)^{3/2}}\left(1+\sum_{k=1}^{N-1}\frac{b_kh^{k}}{(-\zeta_0)^{3k/2}}\right)\right)=\O{}(h^\infty).$$

Then, since 
$$f(\eta)=\eta- \frac{h^{1/3}\W(\eta^2)i}{2}\left(1+\sum_{j=1}^{N-1}\frac{-\zeta_kh^{k}}{\eta^{3k}}-ie^{\frac{4}{3h}i\eta^{3}}\left(1+\sum_{k=1}^{N-1}\frac{b_kh^{k}}{\eta^{3k}}\right)\right)$$
has
\begin{equation}\label{eqn:derBound}|f'(\eta)|\geq c|\zeta_0|h^{-\alpha}\left(1+|\zeta_0|^{1/2}h^{\alpha-1}\right)\,,\quad |f''(\eta)|\leq c|\zeta_0|^2h^{-\alpha-1}\left(1+|\zeta_0|^{1/2}h^{\alpha-1}\right)
\end{equation}
when 
$$|\eta-\zeta_0^{1/2}|\leq Ch.$$
Hence, Lemma \ref{lem:Newton} implies the existence of a solution to $f(\eta)=0$ that is $\O{}(h^\infty)$ close to $(-\zeta_0)^{1/2}.$ Next, by the implicit function theorem, 
$f(\eta)=a(\eta)$ defines $\eta$ as a function of $a$ for $a$ sufficiently small. Thus, since 
$$\O{}\left(h_1h^{-\alpha}(h_1^{N}m^{-N}+\e^N)(1+m^{1/2}h^{\alpha-1}))+h_1^2h^{-\alpha}m^{2/3}\right)=\O{}(h_1^2h^{-\alpha}m^{2/3})$$
there exists a solution, $z(k,h,n)$, to \eqref{eqn:toSolveExact}
with 
$$(-\zeta)^{1/2}=(-\zeta_0)^{1/2}+\frac{a((-\zeta_0)^{1/2},h)}{\partial_\eta f((-\zeta_0)^{1/2})}=(-\zeta_0)^{1/2}+\O{}(h_1^2(1+m^{1/3}h^{\alpha-1})^{-1})=(-\zeta_0)^{1/2}+\O{}(h_1^2).$$
Thus,
$$\zeta=\zeta_0+\O{}((-\zeta_0)^{1/2}h_1^2).$$

This shows that if $m\geq ch^{1-\delta}$, we can solve for $\zeta$ so that 
$$\zeta=\zeta_0+\O{}(h^2)$$
by choosing $N$ large enough.
Now, 
$$\Im (-\zeta_0)=-\frac{3h_1}{8m^{1/3}}\log\left(\frac{4m^{2/3}}{h_1^{2/3}\W^2(m)}+1\right)+\O{}(\e_0h_1m^{-4/3})$$
Hence, we have constructed resonances with 
$$\Im \zeta_1=\frac{3h_1}{8m^{1/3}}\log\left(\frac{4m^{2/3}}{h_1^{2/3}\W^2(m)}+1\right)+\O{}(\e_0h_1m^{-4/3}+h_1^2)$$
Because of the size of the lower order terms above, this construction only gives accurate estimates on $\Im(-\zeta_0)$ when $\delta>(3\alpha-2)/4$.

Thus, for $\delta\geq 0$, there exist resonances coming from modes concentrating  $h^{2/3(1-\delta)}$ close to glancing with 
$$\Im z\sim\begin{cases}Ch^{2\alpha-2/3-\delta/3}& (3\alpha-2)/4< \delta <3\alpha-2\\
Ch&\delta=3\alpha-2\\
Ch^{2/3+\delta/3}\log h^{-1}&3\alpha-2<\delta\leq 1
\end{cases}.$$
Moreover, for each $n$ with $(1-\e)h^{-1}\leq |n|\leq (1+\e)h^{-1}$, we have $(1-Ch^{3/4})nh^{-1}\leq \Re w\leq nh^{-1}(1+Ch^{3/4}).$ Hence, $\Re \zeta$ ranges over an interval of size $Ch^{3/4}$. Together with the construction above, this implies that for each such $n$ we have at least $ch^{-1/4}$ resonances a fixed distance from glancing. Thus, for $M$ large enough
$$\#\{z\in \Lambda(h)\}\geq Ch^{-5/4}.$$
This implies Theorem \ref{thm:lowerBound}.

\bibliography{biblio}

\def\cprime{$'$} \def\cprime{$'$} \def\cprime{$'$}
\begin{thebibliography}{10}

\bibitem{Aligia}
A.~A. Aligia and A.~M. Lobos.
\newblock Mirages and many-body effects in quantum corrals.
\newblock {\em J. Phys.: Condens. Matter}, 17, 2005.

\bibitem{Heller}
M.~Barr, M.~Zaletel, and E.~Heller.
\newblock Quantum corral resonance widths: Lossy scattering as acoustics.
\newblock {\em Nano Letters}, (10):3253--3260, 2010.

\bibitem{Burke}
P.~G. Burke.
\newblock {\em Potential scattering in atomic physics}.
\newblock Plenum Press, New York, 1977.

\bibitem{Bu}
N.~Burq.
\newblock Smoothing effect for {S}chr\"odinger boundary value problems.
\newblock {\em Duke Math. J.}, 123(2):403--427, 2004.

\bibitem{BuZw}
N.~Burq and M.~Zworski.
\newblock Resonance expansions in semi-classical propagation.
\newblock {\em Comm. Math. Phys.}, 223(1):1--12, 2001.

\bibitem{Crommie}
M.~Crommie, C.~Lutz, D.~Eigler, and E.~Heller.
\newblock Quantum corrals.
\newblock {\em Physica D: Nonlinear Phenomena}, 83(1-3):98--108, 1995.

\bibitem{crommie1995waves}
M.~F. Crommie, C.~Lutz, D.~Eigler, and E.~Heller.
\newblock Waves on a metal surface and quantum corrals.
\newblock {\em Surface Review and Letters}, 2(01):127--137, 1995.

\bibitem{LocSmooth1}
K.~Datchev.
\newblock Local smoothing for scattering manifolds with hyperbolic trapped
  sets.
\newblock {\em Comm. Math. Phys.}, 286(3):837--850, 2009.

\bibitem{ZwScat}
S.~Dyatlov and M.~Zworski.
\newblock {\em Mathematical theory of scattering resonances}.

\bibitem{Exner}
P.~Exner.
\newblock Leaky quantum graphs: a review.
\newblock In {\em Analysis on graphs and its applications}, volume~77 of {\em
  Proc. Sympos. Pure Math.}, pages 523--564. Amer. Math. Soc., Providence, RI,
  2008.

\bibitem{Galk}
J.~Galkowski.
\newblock Distribution of resonances in scattering by thin barriers.
\newblock {\em arXiv preprint, arxiv : 1404.3709}, 2014.

\bibitem{GaQV}
J.~Galkowski.
\newblock A quantitative \uppercase{V}ainberg method for black box scattering.
\newblock {\em arXiv preprint, arxiv : 1511.05894}, 2015.

\bibitem{GS}
J.~Galkowski and H.~Smith.
\newblock Restriction bounds for the free resolvent and resonances in lossy
  scattering.
\newblock {\em Int. Math. Res. Not}, 2014.

\bibitem{Lax}
P.~D. Lax and R.~S. Phillips.
\newblock {\em Scattering theory}, volume~26 of {\em Pure and Applied
  Mathematics}.
\newblock Academic Press, Inc., Boston, MA, second edition, 1989.
\newblock With appendices by Cathleen S. Morawetz and Georg Schmidt.

\bibitem{NakSteZw}
S.~Nakamura, P.~Stefanov, and M.~Zworski.
\newblock Resonance expansions of propagators in the presence of potential
  barriers.
\newblock {\em J. Funct. Anal.}, 205(1):180--205, 2003.

\bibitem{NIST}
F.~W.~J. Olver, D.~W. Lozier, R.~F. Boisvert, and C.~W. Clark, editors.
\newblock {\em N{IST} handbook of mathematical functions}.
\newblock U.S. Department of Commerce, National Institute of Standards and
  Technology, Washington, DC; Cambridge University Press, Cambridge, 2010.
\newblock With 1 CD-ROM (Windows, Macintosh and UNIX).

\bibitem{SjoDist}
J.~Sj{\"o}strand and M.~Zworski.
\newblock Complex scaling and the distribution of scattering poles.
\newblock {\em J. Amer. Math. Soc.}, 4(4):729--769, 1991.

\bibitem{TangZw}
S.-H. Tang and M.~Zworski.
\newblock Resonance expansions of scattered waves.
\newblock {\em Comm. Pure Appl. Math.}, 53(10):1305--1334, 2000.

\bibitem{Vain}
B.~R. Va{\u\i}nberg.
\newblock {\em Asymptotic methods in equations of mathematical physics}.
\newblock Gordon \& Breach Science Publishers, New York, 1989.
\newblock Translated from the Russian by E. Primrose.

\bibitem{Vod1}
G.~Vodev.
\newblock Sharp bounds on the number of scattering poles for perturbations of
  the {L}aplacian.
\newblock {\em Comm. Math. Phys.}, 146(1):205--216, 1992.

\bibitem{Vod2}
G.~Vodev.
\newblock Sharp bounds on the number of scattering poles in even-dimensional
  spaces.
\newblock {\em Duke Math. J.}, 74(1):1--17, 1994.

\bibitem{Vod3}
G.~Vodev.
\newblock Sharp bounds on the number of scattering poles in the two-dimensional
  case.
\newblock {\em Math. Nachr.}, 170:287--297, 1994.

\bibitem{ZwAMS}
M.~Zworski.
\newblock Resonances in physics and geometry.
\newblock {\em Notices Amer. Math. Soc.}, 46(3):319--328, 1999.

\bibitem{EZB}
M.~Zworski.
\newblock {\em Semiclassical analysis}, volume 138 of {\em Graduate Studies in
  Mathematics}.
\newblock American Mathematical Society, Providence, RI, 2012.

\end{thebibliography}
\bibliographystyle{abbrv} 
\end{document}